\documentclass[sn-mathphys-num]{sn-jnl}

\usepackage{amsmath,amsfonts,amssymb}
\usepackage{amsthm}
\usepackage{algorithm}
\usepackage{algpseudocode}
\usepackage{graphicx} 
\usepackage{xcolor}
\usepackage{hyperref}

\newtheorem{thm}{Theorem}[section]
\newtheorem{lem}[thm]{Lemma}

\newcommand{\fold}[1]{\text{fold}(#1)}

\date{\today}

\newcommand{\kp}[1]{{k + #1}}
\newcommand{\X}{\mathcal X}
\newcommand{\A}{\mathcal A}
\newcommand{\B}{\mathcal B}

\newcommand{\N}{\mathcal N}

\newcommand{\Y}{\mathcal Y}
\newcommand{\Z}{\mathcal Z}

\newcommand{\norm}[1]{\left\|#1\right\|}
\newcommand{\normF}[1]{\norm{#1}_F}

\newcommand{\expect}{\mathbb E}
\newcommand{\FR}{\Omega}
\newcommand{\projF}[1]{P_{\FR}(#1)}

\newcommand{\proj}[1]{P_{#1}}
\newcommand{\Set}[1]{\left\{#1\right\}}
\newcommand{\compl}{^\mathsf{c}}

\newcommand{\unfold}[1]{\text{unfold}(#1)}
\newcommand{\bdiag}{\textrm{bdiag}}
\newcommand{\bcirc}[1]{\text{bcirc}(#1)}
\newcommand{\dotP}[2]{\left\langle #1, #2 \right\rangle}
\newcommand{\fft}[1]{F(#1)}

\begin{document}
\title[Block Matrix and Tensor Randomized Kaczmarz Methods for Linear Feasibility Problems]
{Block Matrix and Tensor Randomized Kaczmarz Methods for Linear Feasibility Problems}

\author*[1]{\fnm{Minxin} \sur{Zhang}}\email{minxinzhang@math.ucla.edu}
\author[2]{\fnm{Jamie} \sur{Haddock}}\email{jhaddock@g.hmc.edu}
\author[1]{\fnm{Deanna} \sur{Needell}}\email{deanna@math.ucla.edu}

\affil*[1]{\orgdiv{Department of Mathematics}, \orgname{University of California, Los Angeles}, \orgaddress{\city{Los Angeles}, \state{CA}, \country{USA}}}

\affil[2]{\orgdiv{Department of Mathematics}, \orgname{Harvey Mudd College}, \orgaddress{\city{Claremont}, \state{CA}, \country{USA}}}

\abstract{
The randomized Kaczmarz methods are a popular and effective family of iterative methods for solving large-scale linear systems of equations, which have also been applied to linear feasibility problems. In this work, we propose a new block variant of the randomized Kaczmarz method, B-MRK, for solving 
linear feasibility problems defined by matrices. We show that B-MRK converges linearly in expectation to the feasible region.
Furthermore, we extend the method to solve tensor linear feasibility problems defined under the tensor t-product. A tensor randomized Kaczmarz (TRK) method, TRK-L, is proposed for solving linear feasibility problems that involve mixed equality and inequality constraints. Additionally, we introduce another TRK method, TRK-LB, specifically tailored for cases where the feasible region is defined by linear equality constraints coupled with bound constraints on the variables. We show that both of the TRK methods converge linearly in expectation to the feasible region. Moreover, the effectiveness of our methods is demonstrated through numerical experiments on various Gaussian random data and applications in image deblurring.
}
\maketitle
\section{Introduction}
In many real-world scenarios, data naturally occurs as third-order tensors, such as in color images represented by three modes (height, width, color channels), video data organized by frames, rows, and columns, and biomedical signals characterized by time, frequency, and channel.  For data in this form, the tensor t-product is a natural operator that recovers the usual matrix product as a special case, and extends many of the ``nice" properties of this product to the tensor regime.  First introduced in \cite{KILMER2011641}, the tensor t-product has found many applications, such as signal processing~\cite{zhang2016exact,semerci2014tensor}, image processing and encryption~\cite{khalil2021efficient}, image compression and deblurring~\cite{newman2020nonnegative}, machine learning~\cite{settles2007multiple}, computer vision~\cite{zhang2014novel,martin2013order}, and the multi-view clustering problem~\cite{cheng2018tensor}.
These ideas even appear in quantum information processing~\cite{lemm2018multivariate} and are used to build neural network architectures~\cite{newman2018stable}. Recently, there has been significant work developing iterative methods for solving linear systems and linear regression problems defined under the t-product; see e.g,~\cite{kilmer2013third,ma2022randomized,huang2023tensor}.

\subsection{Randomized Kaczmarz for linear equations}
The randomized Kaczmarz (RK) methods are a commonly used family of iterative methods for solving large linear systems of the form $Ax=b,$ where $A$ is an $m\times n$ matrix and $b$ is an $m$-dimensional vector. These methods iterate by randomly sampling one of the rows of the matrix $A$, and projecting the current iterate, $x^{k}$, onto the hyperplane solution set of the equation defined by the sampled row.  The standard RK update is given by
\[
x^\kp1 = x^k - a_{i_k}\frac{a_{i_k}^T x^k-b_{i_k}}{\norm{a_{i_k}}^2},
\]
where $i_k$ is the row index sampled at iteration $k$ and $a_{i_k}^T$ is the $i_k$-th row of $A$. For a consistent linear system,
Strohmer and Vershynin \cite{strohmer2006randomized, Strohmer2007ARK} established that, when selecting the row indices $i$ in each iteration independently at random with probabilities 
proportional to the squared norms $\norm{a_{i}}^2$, the RK method converges at least linearly in expectation to the solution $x^*$ of the system, with the guarantee
\[
\expect[\norm{x^\kp1-x^*}^2 | x^k]\le (1-\kappa(A)^{-2})\norm{x^k-x^*}^2,
\]
where $\kappa(A):=\normF{A}\norm{A^{-1}}_2$ is the \emph{scaled condition number}~\cite{demmel1988probability}, with 
$\normF{\cdot}$ being the Frobenius norm and $A^{-1}$ the left inverse of $A$.

Due to a block-iterative scheme introduced in \cite{elfving1980block}, a block variant of 
the RK methods was first developed and analyzed by Needell and Tropp \cite{needell2014paved}. 
The method first creates a partition $T:=\Set{\tau_1,\cdots,\tau_{m'}}$ of the row indices. For $\tau\in T$, let 
$A_\tau$ and $b_\tau$ be the submatrix of $A$ and the subvector of $b$ indexed by $\tau$ respectively. 
The partition $T$ is called an 
\emph{$(m', \alpha, \beta)$ row paving} of the matrix $A$ if
\[
\lambda_{\min} (A_\tau A_\tau^*)\ge \alpha \quad\textrm{and}\quad \lambda_{\max}(A_\tau A_\tau^*)\le \beta\quad\textrm{for each } \tau\in T,
\]
where $\lambda_{\min}$ and $\lambda_{\max}$ denote the minimum and maximum eigenvalues respectively.
In each iteration, a random block $\tau\in T$ is selected and the current iterate $x^k$ is projected onto the solution set 
$\Set{x:A_\tau x=b_\tau}$ using the update rule
\begin{equation}\label{blockRKOG}
x^\kp1 = x^k-(A_\tau)^\dag(A_\tau x^k-b_\tau),
\end{equation}
where $(A_\tau)^\dag$ is the (Moore-Penrose) pseudoinverse of $A_\tau.$ It was shown that the method satisfies
\[
\expect{\norm{x^k-x^*}^2}\le\left[1-\frac{\sigma^2_{\min}(A)}{\beta m'}\right]^k\norm{x^0-x^*}^2+\frac{\beta}{\alpha}\frac{\norm{Ax^*-b}^2}{\sigma^2_{\min}(A)},
\]
where $x^*$ is the least-squares solution of the linear system, and $\sigma_{\min}$ denotes the minimum singular value.
Hence, for consistent linear systems, the block RK method converges linearly in expectation, with the rate of convergence dependent on
the row paving.  In~\cite{HJY21}, the authors extended the convergence analysis for block RK to a broader class of sets of blocks which include row pavings as a special case. 

The RK methods have also been extended to the tensor setting under t-product multiplication \cite{ma2022randomized,chen2021regularized, bao2022randomized, zhangsampling, liao2024accelerated}.
An equivalence between a tensor RK method and a block variant of RK methods for matrices was established in the Fourier domain
in \cite{ma2022randomized}.

\subsection{Randomized Kaczmarz for linear equality and inequality constraints}
As a generalization of the RK methods for linear equations, Leventhal and Lewis~\cite{leventhal2010randomized} proposed a randomized algorithm for solving linear feasibility problems of the form
\begin{equation}\label{eq:matrixIeq}
\begin{cases}
a_i^Tx\le b_i & (i \in I_\le) \\
a_i^Tx= b_i & (i \in I_=),
\end{cases}
\end{equation}
where $a_i\in \mathbb R^n$ for each i, and the disjoint index sets $I_\le$ and $I_=$ partition the set $\Set{1,2,\cdots,m}$.
At each iteration $k,$ the algorithm randomly samples an index $i_k\in I_\le\cup I_=$ and, if $i_k\in I_=,$ or if $i_k\in I_\le$ and 
$a_{i_k}^Tx> b_{i_k}$, projects the current iterate
$x^k$ onto the hyperplane defined by $\Set{x:a_{i_k}^Tx= b_{i_k} }$.
It was shown that the algorithm converges at least linearly in expectation, with the guarantee
\[
\expect[d(x^\kp1,\FR)^2 | x^k]\le \left(1-\frac{1}{L^2\normF{A}^2}\right)d(x^k,\FR)^2 ,
\]
where $\FR$ is the feasible region defined by (\ref{eq:matrixIeq}), $d(x,\FR) := \min_{y \in \FR} \|x - y\|$ denotes the distance of a point $x$ to $\FR$,
and $L$ is the Hoffman constant \cite{hoffman1952}. The Hoffman constant is defined as the smallest constant $L$ such that, for all $x\in\mathbb R^n,$
\begin{equation}\label{eq:hoffman0}
d(x,\FR)\le L\norm{c(Ax-b)},
\end{equation}
where the function $c:\mathbb R^m\to\mathbb R^m$ is defined by
\begin{equation}\label{cfunc}
c(y)_i = \begin{cases}
\max\{y_i,0\} & (i \in I_\le) \\
y_i & (i \in I_=).
\end{cases}
\end{equation}
Several authors have provided geometric or algebraic interpretations of the Hoffman constant, e.g.,
\cite{guler1995approximations, zheng2004hoffman, pang1997error}; or have provided algorithmic approaches to compute or bound the constant~\cite{pena2024easily}.
In the case of linear equations (i.e., $I_\le=\emptyset$), by singular value decomposition, $L$ is simply the reciprocal of the
smallest nonzero singular value of $A$, and hence equals the norm of the left inverse, $\norm{A^{-1}}_2$, when $A$ has full rank.

A block variant of the RK method for solving (\ref{eq:matrixIeq}) was proposed in \cite{briskman2015block}. The method first creates a partition 
$T = \Set{\tau_1,\cdots\tau_{m'}}$ of row indices in $I_\le$ and a partition $T'=\Set{\tau_{1+m'},\cdots\tau_{m''}}$ of row indices in $I_=$. In each iteration, a random block $\tau\in T\cup T'$ is selected and the current iterate  is updated by
\[
x^\kp1 = \begin{cases}
x^k-(A_\tau)^\dag(A_\tau x^k-b_\tau)_+ & (\textrm{if }\tau\in T)\\
x^k-(A_\tau)^\dag(A_\tau x^k-b_\tau) & (\textrm{if }\tau\in T'),
\end{cases}
\]
where $(A_\tau x^k-b_\tau)_+=\max\{A_\tau x^k-b_\tau,0\},$ with the maximum operation applied entrywise.
It was shown that the method converges linearly in expectation if $T$ is an \emph{obtuse row paving}, which is a special
row paving that satisfies a certain geometry property. In general, obtuse row pavings are challenging to obtain.

\subsection{Contributions}
In this work, we propose a new block RK method, B-MRK, for solving feasibility linear systems of the form
\[
\begin{cases}
a_i^T X\le b_i^T & (i \in I_\le) \\
a_i^T X= b_i^T & (i \in I_=),
\end{cases}
\]
where $a_i^T$ represents the $i$-th row of a matrix $A\in \mathbb R^{m\times n}$, 
$b_i^T$ represents the $i$-th row of a matrix $B \in\mathbb R^{m\times p}$, $X\in \mathbb R^{n\times p}$,
and the disjoint index sets $I_\le$ and $I_=$ partition the set $\Set{1,2,\cdots,m}$. 
The method does not require the calculation of the pseudoinverse of each block matrix $A_\tau,$
making it computationally cheaper and numerically more stable than the block method described in \eqref{blockRKOG}. Moreover, we show it converges linearly in expectation without reliance on any special row paving.

Furthermore, we extend the method to solve linear feasibility problems defined
under the tensor t-product~\cite{KILMER2011641}. For third-order tensors $\A\in\mathbb R^{m\times l\times n}$, $\X\in\mathbb R^{l\times p\times n}$ and $\B\in\mathbb R^{m\times p\times n}$, we consider the following 
linear equality and inequality constraints:
\begin{equation}\label{eq:lc}
\begin{cases}
\mathcal A_{i::}\ast\mathcal X \le \mathcal B_{i::} & (i \in I_\le) \\
\mathcal A_{i::}\ast\mathcal X = \mathcal B_{i::} & (i \in I_=),
\end{cases}
\end{equation}
where the symbol  “$\ast$” denotes the tensor t-product (defined below in \eqref{def:tprod}
), $\mathcal A_{i::}:=\mathcal A(i,:,:)$ represents the $i$-th row slice of $\A$,
$\mathcal B_{i::}:=\mathcal B(i,:,:)$ represents the $i$-th row slice of $\B$,
and the disjoint index sets $I_\le$ and $I_=$ partition the set $I=\{1,2,...,m\}$. 
We propose a tensor RK method, TRK-L, for solving linear feasibility problems of the above form. 
TRK-L iterates by independently sampling a random index $i_k$ and then projecting the current iterate $\X_k$ towards the solution set
given by $\Set{\X:\mathcal A_{i_k::}\ast\mathcal X \le \mathcal B_{i_k::} }$ if $i_k\in I_\le$ or 
$\Set{\X:\mathcal A_{i_k::}\ast\mathcal X = \mathcal B_{i_k::} }$ if $i_k\in I_=$.
A bounded step size $t_{i_k}$ is incorporated at each iteration $k$ to ensure convergence.
To bound convergence, the notion of a Hoffman constant is extended to the tensor setting. 
We show that TRK-L converges linearly in expectation to the feasible region, with the rate of convergence 
given in terms of the Hoffman constant.
Additionally, we introduce another TRK method, TRK-LB, specifically tailored for cases where the feasible region is defined by linear equality constraints coupled with bound constraints on the variables:
\begin{equation}\label{eq:lb}
\A\ast\X = \B, ~\X\le \tilde\B,
\end{equation}
where $\A\in\mathbb R^{m\times l\times n}$, $\X\in\mathbb R^{l\times p\times n}$, $\B\in\mathbb R^{m\times p\times n}$
and $\tilde\B\in\mathbb R^{l\times p\times n}$.  Nonnegativity constraints are a special case of this type of bound constraint.
At each iteration of TRK-LB, a random index $i_k$ is selected and the current iterate $\X_k$ is projected towards
$\Set{\X:\mathcal A_{i_k::}\ast\mathcal X = \mathcal B_{i_k::} }$ with respect to a step size $t_{i_k}$, followed by the orthogonal
projection onto the polyhedral set $\Set{\X\le \tilde\B}.$ Note that the orthogonal projection of a point $\X$ onto  $\Set{\X\le \tilde\B}$
can be easily computed as $\min\{\X,\tilde\B\}$, with the minimum operation applied to $\X$ entrywise.
We show that TRK-LB also converges linearly in expectation to the feasible region. 

We demonstrate the effectiveness of B-MRK, TRK-L and TRK-LB through numerical experiments
on various Gaussian random data, as well as applications in image deblurring, where non-negativity constraints are 
naturally applied.

\subsection{Organization}
The rest of the paper is organized as follows. In Section~\ref{sec:preliminaries}, we present basic concepts and results of third-order tensors. In Section~\ref{sec:b-mrk}, we formulate the block RK method, B-MRK, for solving linear feasibility problems defined by matrices, and show its linear convergence in expectation.
In Section~\ref{sec:trkL}, we
propose the tensor RK method, TRK-L, for solving linear feasibility problems of the general form 
(\ref{eq:lc}), and show its linear convergence in expectation. In Section~\ref{sec:comp}, we make analytical connections 
between the TRK-L method and the B-MRK method.
In Section~\ref{sec:trklb}, we
introduce the TRK-LB method for solving problems of the special form 
(\ref{eq:lb}) and also show its linear convergence in expectation. 
We showcase numerical experiments in Section~\ref{sec:exp} and conclude in Section~\ref{sec:conclude}.

\section{Preliminaries}\label{sec:preliminaries}
Throughout this work, we use calligraphic capital letters to represent tensors and regular capital letters to represent matrices.
For a third-order tensor $\A\in\mathbb R^{m\times l\times n}$, the \emph{block circulant operator} is defined as follows:
\[
\bcirc{\A} := \begin{bmatrix}
    A_1 & A_n & \cdots & A_2 \\
    A_2 & A_1 & \cdots & A_3 \\
    \vdots & \vdots & \ddots & \vdots \\
    A_n & A_{n-1} & \cdots & A_1
\end{bmatrix}
\]
where $A_k := \A(;,:,k)$ is the $k$-th \emph{frontal slice} of $\A$.
For the conversion between tensors and matrices, we define the operator $\unfold{\cdot}$ and its inversion $\fold{\cdot}$:
\[
\text{unfold}(\A) = 
\begin{bmatrix}
    A_1 \\
    A_2 \\
    \vdots \\
    A_n
\end{bmatrix}
\in \mathbb{R}^{mn\times l}, \quad 
\text{fold} \left( 
\begin{bmatrix}
    A_1 \\
    A_2 \\
    \vdots \\
    A_n
\end{bmatrix}
\right) = \A.
\]
The transpose of $\A,$ denoted by $\A^T,$ is the $l\times m\times n$ tensor obtained by transposing each of the frontal slices
of $\A$ and then reverse the order of transposed frontal slices $2$ through $n.$ Thus,
\[
\bcirc{\A^T} = \bcirc{\A}^T.
\]
The Frobenius norm of $\A$ is defined as
\[
\normF{\A} :=  \sqrt{\sum_{i,j,k}\A(i,j,k)^2}
\]
and we define the distance in terms of this norm, $d(\X, \FR) := \min_{\Y \in \FR} \|\X - \Y\|_F$.

For two tensors $\A$ and $\tilde\A$ of the same size, we define the inner product as
\[
\langle\A,\tilde\A\rangle = \sum_{i,j,k} \A(i,j,k)\tilde\A(i,j,k).
\]
Hence,
$\normF{\A}^2 = \langle\A, \A\rangle $.
If $\A\in\mathbb R^{m\times l\times n}$ and $\X\in\mathbb R^{l\times p\times n}$,
their \emph{tensor t-product} is defined as the $m\times p\times n$ given by
\begin{equation}\label{def:tprod}
\A\ast\X:= \fold{\bcirc{A}\unfold{\X}}.
\end{equation}
The t-product is separable in the first dimension:
\[
\A_{i::}\ast\X = (\A\ast\X)(i,:,:),
\]
where $\A_{i::}:=\A(i,:,:)$ denotes the $i$-th row slice of $\A$ for each index $i$.
It also holds that
\[
\A\ast\X = \sum_{j=1}^l \A(:,j,:)\X(j,:,:).
\]
Moreover, for a tensor $\B\in\mathbb R^{m\times p\times n} $, by \cite[Lemma 2.7]{chen2021regularized} we have 
\[
\langle\A*\X,\B\rangle=\langle\X,\A^T*\B\rangle.
\]

With the block circulant matrix involved, the t-product can be efficiently implemented using the discrete Fourier transform (DFT).
Given a tensor $\A\in\mathbb R^{m\times l\times n}$, the DFT of $\A$ is defined as a tensor $\fft{\A}$ that is obtained by taking the DFT along each tube fiber $\A(i,j,:)$ of $\A$, i.e., for each $i\in\Set{1,\cdots,m}$ and $j\in\Set{1,\cdots,l}$,
\begin{equation}\label{eq:fft}
\fft{\A}(i,j,:) = \sqrt{n}F_n\A(i,j,:),
\end{equation}
where $F_n$ denotes the $n\times n$ unitary DFT matrix. 
It follows immediately that 
\begin{equation}\label{eq:Fnorm}
\normF{\fft{\A}}^2 = n\normF{\A}^2.
\end{equation}
For each $k$, we define $\fft{\A}_k$ as the $k$-th frontal slice of $\fft{\A}$, i.e. $$\fft{\A}_k := \fft{\A}(:,:,k).$$
Block circulant matrices can be block diagonalized by DFT as follows:
\begin{equation}\label{eq:bdiag}
\bdiag(\fft{\A}) := (F_n\otimes I_m)\bcirc{\A}(F_n^*\otimes I_l)=\left(\begin{array}{cccc}
\fft{\A}_{1} & & &\\
 & \fft{\A}_{2} & &\\
& & \ddots & \\
& & & \fft{\A}_{n}
\end{array}\right),
\end{equation}
where $\otimes$ is the Kronecker product.

We let $(\A)_+ = \max(\A,0)$ with the maximum operation applied to $\A$ entry-wise.
We also define a function $c_T:\mathbb R^{l\times p\times n}\to \mathbb R^{l\times p\times n}$, analogous to \eqref{cfunc}, by
\begin{equation}\label{eq:cTfunc}
c_T(\Z)(i,j,k) = 
\begin{cases}
\max\left\{\Z(i,j,k) ,0\right\} & (i \in I_\le)\\
\Z(i,j,k) & (i \in I_=).
\end{cases}
\end{equation}

\section{A Block MRK Method for Linear Constraints}\label{sec:b-mrk}
In this section, we propose a block RK method, B-MRK, for linear feasibility problems defined by matrices.
Let $A\in \mathbb{R}^{m\times n}$, $X \in \mathbb{R}^{n\times p}$ and $B \in \mathbb{R}^{m \times p}$. We consider problems of the following form:
\begin{equation}\label{eq:mLC}
\begin{cases}
a_i^T X\le b_i^T & (i \in I_\le) \\
a_i^T X= b_i^T & (i \in I_=),
\end{cases}
\end{equation}
where $a_i^T$ represents the $i$-th row of $A$, $b_i^T$ represents the $i$-th row of $B$, and the disjoint index sets $I_\le$ and $I_=$ partition the set $\Set{1,2,\cdots,m}$. 

\begin{algorithm}[!htp]
\caption{B-MRK for matrix linear constraints}\label{alg:block_mrk}
\begin{algorithmic}[1]
\State \textbf{Input:} \begin{itemize}
\item $X^0 \in \mathbb{R}^{n\times p}$, $A\in \mathbb{R}^{m\times n}$, $B \in \mathbb{R}^{m \times p}$, indices partition $I_\le$ and $I_=$ of $\Set{1,\cdots,m}$
\item partition $T:=\Set{\tau_1,\cdots\tau_{m'}}$ of $I_\le$ and partition $T':=\Set{\tau_{1+m'},\cdots\tau_{m''}}$ of $I_=$
\item probabilities $p_{\tau_i} = \|A_{\tau_i}\|_F^2/\|A\|_F^2$ for $i\in\Set{1,\cdots,m''}$
\item step sizes $t_{\tau_i}<2$
\end{itemize}
\For{$k = 0,1,2,\dots$}
    \State Sample $\tau\sim \Set{p_{\tau_i}}$
    \If{$\tau\in T$}
    \State $X^{k+1} = X^k - t_{\tau}A_{\tau}^T\left( A_{\tau}X^k - B_{\tau} \right)_+/ \|A_{\tau}\|_F^2$
    \Else
    \State $X^{k+1} = X^k - t_{\tau}A_{\tau}^T\left( A_{\tau}X^k - B_{\tau} \right)/ \|A_{\tau}\|_F^2$
    \EndIf
\EndFor
\State \textbf{Output:} last iterate $X^{k}$
\end{algorithmic}
\end{algorithm}

Details of the B-MRK method are summarized in Algorithm~\ref{alg:block_mrk}. The method starts with
a partition $T:=\Set{\tau_1,\cdots\tau_{m'}}$ of the index set $I_\le$ and a partition $T':=\Set{\tau_{1+m'},\cdots\tau_{m''}}$ of the index set $I_=$. In each iteration, a block 
$\tau\in T\cup T'$ is independently selected at random with probabilities 
proportional to the squared Frobenius norms $\normF{A_\tau}^2$. If $\tau\in T$, the current iterate
$X^k$ is projected towards the set $\{X: A_\tau X \le B_\tau\}$ in the coordinates where $X^k$ violates the entrywise inequalities, that is, the support of $\left( A_\tau X^k - B_\tau \right)_+$. Otherwise, $X^k$ is projected towards the set $\{X: A_\tau X = B_\tau\}$.
A step size $t_\tau$ is used. Below, we show that the
B-MRK method converges linearly in expectation if choosing $t_\tau<2$ for all $\tau\in T\cup T'$.

The Hoffman constant $L$ in (\ref{eq:hoffman0}) is defined for vectors $x\in\mathbb R^n$. The following lemma extends this definition to matrices $X\in\mathbb R^{n\times p}$.
\begin{lem}\label{lem:hoffmanM}
Suppose the system (\ref{eq:mLC}) has a nonempty feasible region $\FR$. Then there exists a smallest constant $L>0$ such that, for all $X\in\mathbb R^{n\times p},$
\begin{equation}\label{eq:hoffmanM}
d(X,\FR)\le L\normF{ c_M(AX-B)},
\end{equation}
where the function $ c_M:\mathbb R^{m\times p}\to\mathbb R^{m\times p}$ is defined by
\begin{equation}\label{eq:cMfunc}
 c_M(Y)_{ij} = \begin{cases}
\max\{Y_{ij},0\} & (i \in I_\le) \\
Y_{ij} & (i \in I_=).
\end{cases}
\end{equation}
We refer to $L$ as the Hoffman constant in the tensor setting.
\end{lem}
\begin{proof}
The matrix $X$ can be converted to a vector $x\in\mathbb{R}^{np}$ such that $X_{kj} = x_{(j-1)n+k}$ for $1\le k\le n$ and $1\le j\le p.$
Similarly, the matrix $B$ can be vectorized to $b\in\mathbb R^{mp}$ such that $B_{ij} = b_{(j-i)n+i}$ for $1\le i\le m$ and $1\le j\le p.$
The system (\ref{eq:mLC}) can be equivalently written as
\begin{equation}\label{eq:vFR}
\begin{cases}
a_i^T x_{((j-1)n+1):(jn)}\le b_{(j-i)n+i} & (i \in I_\le, j\in\{1,\cdots,p\}) \\
a_i^T x_{((j-1)n+1):(jn)} = b_{(j-i)n+i} & (i \in I_=, j\in\{1,\cdots,p\}).
\end{cases}
\end{equation}
Hence, there exists the Hoffman constant $L>0$ such that, for all $x\in\mathbb{R}^{np}$,
\[
d(x,\tilde\FR)\le L\norm{c(\tilde Ax-b)},
\]
where $\tilde\FR\subset\mathbb R^{np}$ is the feasible region defined by (\ref{eq:vFR}), $c$ is the function defined by (\ref{cfunc}), 
and $\tilde A:=\textrm{diag}\{A,\cdots,A\}\in\mathbb R^{mp\times np}.$
As $x$ is the vectorization of $X$ and (\ref{eq:vFR}) is equivalent to (\ref{eq:mLC}), $d(x,\tilde\FR) = d(X,\FR).$
Moreover, it can be easily verified that $(\tilde Ax-b)$ is the vectorization of $(AX-B).$ Therefore,
\[
d(X,\FR) = d(x,\tilde\FR)\le L\norm{c(\tilde Ax-b)}= L\normF{ c_M(AX-B)}.
\]
The proof is thus completed.
\end{proof}

\begin{thm}\label{thm:mRKconverge}
Suppose the system (\ref{eq:nC}) has a nonempty feasible region $\FR$. Algorithm~\ref{alg:block_mrk} generates a sequence of iterations $\Set{\X^k}_{k\ge 0}$ 
that converges linearly in expectation:
\begin{equation}\label{eq:mRKconverge}
\expect [d(X^\kp1,\FR)^2 | X^k] \le \left(1-\min_{\tau}2t_\tau\left(1-\frac{t_\tau}{2}\right)\frac{1}{L^2\normF{A}^2}\right)d(X^k,\FR)^2,
\end{equation}
where $L$ is the Hoffman constant given by (\ref{eq:hoffmanM}).
\end{thm}
\begin{proof}
Let $\projF{\cdot}$ denote the projection onto the feasible region $\FR$. Then, for $\tau\subset I_\le,$
\begin{align*}
\normF{X^\kp1-\projF{X^\kp1}}^2  \le & \normF{X^\kp1-\projF{X^k}}^2\\
= & \normF{X^k - \frac{t_{\tau}A_{\tau}^T\left( A_{\tau}X^k - B_{\tau} \right)_+}{\|A_{\tau}\|_F^2}-\projF{X^k}}^2\\
\le & \normF{X^k-\projF{X^k}}^2 + t_\tau^2\frac{\normF{(A_{\tau}X^k - B_{\tau})_+}^2}{\normF{A_\tau}^2}\\
& - 2\dotP{X^k-\projF{X^k}}{\frac{t_{\tau}A_{\tau}^T\left( A_{\tau}X^k - B_{\tau} \right)_+}{\|A_{\tau}\|_F^2}}\\
\le &  \normF{X^k-\projF{X^k}}^2 +  t_\tau^2\frac{\normF{(A_{\tau}X^k - B_{\tau})_+}^2}{\normF{A_\tau}^2}\\
& -\frac{2t_\tau\normF{(A_\tau X^k - B_{\tau})_+}^2}{\|A_{\tau}\|_F^2}\\
= & \normF{X^k-\projF{X^k}}^2 - \frac{2t_\tau}{\normF{A_\tau}^2}(1-\frac{t_\tau}{2})\normF{(A_\tau X^k - B_{\tau})_+}^2.
\end{align*}
Similarly, for $\tau\subset I_=,$
\[
\normF{X^\kp1-\projF{X^\kp1}}^2  \le\normF{X^k-\projF{X^k}}^2 - \frac{2t_\tau}{\normF{A_\tau}^2}(1-\frac{t_\tau}{2})\normF{(A_\tau X^k - B_{\tau})}^2.
\]

Taking the expectation with respect to the specified probability distribution gives
\[
\expect [d(X^\kp1,\FR)^2 | X^k] \le d(X^k,\FR)^2-\min_{\tau}2t_\tau\left(1-\frac{t_\tau}{2}\right) \frac{\normF{c_M(AX^k-B)}^2}{\normF{A}^2}.
\]
Then by Lemma~\ref{lem:hoffmanM},
\[
\expect [d(X^\kp1,\FR)^2 | X^k] \le \left(1-\min_{\tau}2t_\tau\left(1-\frac{t_\tau}{2}\right)\frac{1}{L^2\normF{A}^2}\right)d(X^k,\FR)^2,
\]
where $L$ is the Hoffman constant given by (\ref{eq:hoffmanM}).
\end{proof}

\section{A TRK Method for Linear Constraints}\label{sec:trkL}
Suppose that $\mathcal A\in\mathbb R^{m\times l\times n}$, $\mathcal X\in\mathbb R^{l\times p\times n}$ and $\mathcal B\in\mathbb R^{m\times p\times n}$.
In this section, we consider the tensor linear feasibility problem given by the constraints
\begin{equation}\label{eq:LC}
\begin{cases}
\mathcal A_{i::}\ast\mathcal X \le \mathcal B_{i::} & (i \in I_\le) \\
\mathcal A_{i::}\ast\mathcal X = \mathcal B_{i::} & (i \in I_=),
\end{cases}
\end{equation}
where $I_\le$ and $I_=$ are disjoint index sets that partition the set $I=\Set{1,\cdots,m}$, and the inequality $\mathcal A_{i::}\ast\mathcal X \le \mathcal B_{i::}$ holds entrywise.  We assume that
the system (\ref{eq:LC}) has a nonempty feasible region $\FR$.

In Algorithm~\ref{alg:trklc}, we propose a variant of the TRK method, TRK-L, for solving the system (\ref{eq:LC}).
The method projects towards the boundary of the solution set of the sampled constraint in the coordinates in which it is violated.  That is, if $i_k$, the sampled index in the $k$th iteration, indexes an equality constraint, Algorithm~\ref{alg:trklc} projects towards the set of solutions to this constraint, $\{\X: \A_{i_k ::} \ast \X = \B_{i_k ::}\}$.  Otherwise, if $i_k$ indexes an inequality constraint, then Algorithm~\ref{alg:trklc} projects towards the solution set of this inequality $\{\X: \A_{i_k ::} \ast \X \le \B_{i_k ::}\}$ in the coordinates where $\X^k$ does not already satisfy this entrywise inequality, that is, the support of $\left( \mathcal{A}_{i_k::}\ast\mathcal{X}^k - \mathcal{B}_{i_k::} \right)_+$.  
\begin{algorithm}[!htp]
\caption{TRK-L for tensor linear constraints}\label{alg:trklc}
\begin{algorithmic}[1] 
\State \textbf{Input:} \begin{itemize}
\item $\mathcal X^0 \in \mathbb{R}^{l\times p \times n}$, $\mathcal A\in \mathbb{R}^{m\times l \times n}$, $\mathcal{B} \in \mathbb{R}^{m \times p \times n}$, 
indices partition $I_\le$ and $I_=$, 
 \item probabilities $p_i = \|\mathcal A_{i::}\|_F^2/\|\mathcal A\|_F^2$ and 
 \item step sizes 
 $t_i<2\normF{\A_{i::}}^2/(\max_{1\le j\le n}{\normF{\fft{\A_{i::}}_j}^2})$ for $i=1,2,\cdots,m$
 \end{itemize}
\For{$k = 0,1,2,\dots$}
    \State Sample $i_k \sim \Set{p_i}$
    \If{$i\in I_\le$}
    \State $\mathcal{X}^{k+1} = \mathcal{X}^k - t_{i_k}\mathcal{A}_{i_k::}^T\ast\left( \mathcal{A}_{i_k::}\ast\mathcal{X}^k - \mathcal{B}_{i_k::} \right)_+/
    \|\mathcal A_{i_k::}\|_F^2$
    \Else
    \State $\mathcal{X}^{k+1} = \mathcal{X}^k - t_{i_k}\mathcal{A}_{i_k::}^T \ast \left( \mathcal{A}_{i_k::}\ast\mathcal{X}^k - \mathcal{B}_{i_k::} \right)
    /\|\mathcal A_{i_k::}\|_F^2$
    \EndIf
\EndFor
\State \textbf{Output:} last iterate $\mathcal{X}^{k+1}$
\end{algorithmic}
\end{algorithm}

We will show that the TRK-L method described in Algorithm~\ref{alg:trklc} converges linearly in expectation to the feasible region
$\FR.$ To this end, we first prove two lemmas.
\begin{lem}\label{lem1}
Suppose the system (\ref{eq:LC}) has a nonempty feasible region $\FR$.
For a point $\X\notin\FR$, let $\Y$ be the projection of $\X$ onto $\FR$, i.e.
$\Y$ is the closest point to $\X$ in $\FR$ with respect to $\|\cdot\|_F$.
Define the subset of indices of active inequality constraints for $\Y$,
\begin{equation}\label{set:indices}
S:=\Set{(i,j,k): i\in I_\le \textrm{ and } (\A\ast\Y)(i,j,k)=\B(i,j,k)},
\end{equation}
and define the set
$$\FR_{S}:=\left\{\Z: (\A\ast\Z)(i,j,k)\le\B(i,j,k) \textrm{ for } (i,j,k)\in S ,
~\A_{i::}\ast\Z= \mathcal B_{i::} \textrm{ for } i\in I_=\right\}.$$
Then $\X\notin\FR_S$ and $\Y$ is also the projection of $\X$ onto the set $\FR_S$.
\end{lem}
\begin{proof}
We first prove by contradiction that $\X\notin\FR_S$. Suppose $\X\in\FR_S$, then there exists $(i,j,k)\notin S$
such that $(\A\ast\X)(i,j,k)>\B(i,j,k)$, while $(\A\ast\Y)(i,j,k)<\B(i,j,k)$ for all $(i,j,k)\notin S$ with $i\in I_\le$. 
Thus, for each $(i,j,k)\notin S$, there exists $\tau_{(i,j,k)}\in (0,1)$ such that $\Y_{\tau_{(i,j,k)}}:=\tau_{(i,j,k)}\Y+(1-\tau_{(i,j,k)})\X$ satisfies $(\A\ast\Y_{\tau_{(i,j,k)}})(i,j,k) \le \B(i,j,k)$.  Define $\tau = \min_{(i,j,k) \not\in S} \tau_{(i,j,k)}$ and note that $\Y_{\tau}:=\tau\Y+(1-\tau)\X$ satisfies $\Y_\tau\in\FR$ and $\Y_\tau$ is closer to $\X$ than $\Y$, which contradicts the assumption that 
$\Y$ is the projection of $\X$ onto $\FR$. Hence, $\X\notin\FR_S$. 

It remains to show that $\Y$ is the projection of $\X$ onto $\FR_S$.
For arbitrary $\Y'\in\FR_S$, it's obvious that $\Y'_\tau:=\tau\Y+(1-\tau)\Y'\in\FR_S$ 
for all $\tau\in [0,1]$. Since $(\A\ast\Y)(i,j,k)<\B(i,j,k)$ for all $(i,j,k)\notin S$ with $i\in I_\le$,
there exists $\tau>0$ sufficiently small such that $\Y'_\tau\in \FR$. It follows that 
\[
\normF{\X-\Y'_\tau}\le \tau\normF{\X-\Y}+(1-\tau)\normF{\X-\Y'}.
\]
By assumption, $\normF{\X-\Y'_\tau}\ge\normF{\X-\Y}$, which implies that 
$\normF{\X-\Y'}\ge \normF{\X-\Y}$. Hence $\Y$ is indeed the closest point to $\X$ in $\FR_S$,
i.e. $\Y$ is the projection of $\X$ onto the set $\FR_S$.
\end{proof}

For the set of indices $S$ given in (\ref{set:indices}), define the function 
$c_S:\mathbb R^{m\times p\times n}\to \mathbb R^{m\times p\times n}$ by
\[ c_S(\Z)(i,j,k) = 
\begin{cases}
\max\left\{\Z(i,j,k) ,0\right\} & ((i,j,k) \in S)\\
0 & (i \in I_\le \textrm{ and } (i,j,k)\notin S)\\
\Z(i,j,k) & (i \in I_=).
\end{cases}
\]

\begin{lem}\label{lem2}
Let $S$ be the set of indices given in (\ref{set:indices}).
Define a set $C$ by 
\[
C:=\left\{\Z: (\A\ast\Z)(i,j,k)\le 0 \textrm{ for } (i,j,k)\in S ,
~\A_{i::}\ast\Z= 0 \textrm{ for } i\in I_=\right\},
\]
and define the set $E$ by
\[
E:=\Set{\Z\in C\compl: \proj{C}(\Z) = 0 },
\]
where $\proj{C}(\cdot)$ denotes the projection onto the set $C$.
Then there exists a constant $\gamma_S>0$ such that for all $\Z\in E$,
\[
\normF{c_S(\A\ast\Z)}\ge\gamma_S\normF{\Z}.
\]
\end{lem}
\begin{proof}
It follows from the homogeneity of $c_S(\A\ast\Z)$ that
\[
\min_{\Z\in E} \frac{\normF{c_S(\A\ast\Z)}}{\normF{\Z}} = \min_{\Z\in E} \normF{c_S(\A\ast\frac{\Z}{\normF{\Z}})}=\min_{\Z\in E, \normF{\Z}=1}\normF{c_S(\A\ast\Z)}.
\]
By the compactness of the set $\Set{\Z\in E: \normF{\Z}=1}$ and the fact that $\normF{c_S(\A\ast\Z)}> 0$ for all $\Z\in E$,
\[
\gamma_S:=\min_{\Z\in E, \normF{\Z}=1}\normF{c_S(\A\ast\Z)}> 0.
\]
Therefore,
\[
\normF{c_S(\A\ast\Z)}\ge\gamma_S\normF{\Z} \textrm{ for all } \Z\in E.
\]
\end{proof}

The following theorem is an extension of a result by Hoffman in \cite{hoffman1952}.
\begin{thm}\label{thm:dist}
Suppose the system (\ref{eq:LC}) has nonempty feasible region $\FR$. There exists a constant $\gamma>0$ such that for any 
$\X\in\mathbb R^{l\times p\times n}$,
\begin{equation}\label{eq:hoffman}
d(\X,\FR)\le \gamma\normF{c_T(\A\ast\X-\B)}.
\end{equation}
\end{thm}
\begin{proof}
For $\X\in\FR$, it is obvious that
\[
d(\X,\FR) = 0 =\normF{c_T(\A\ast\X-\B)}.
\]
For $\X\notin\FR$, let $\Y$ be the projection of $\X$ onto $\FR$. Let $S$ be the subset of indices given in (\ref{set:indices}),
and let $\FR_S$ be the set defined as in Lemma~\ref{lem1}.
Then by Lemma~\ref{lem1}, $\X\not\in\FR_S$ and $\Y$ is also the projection of $\X$ onto $\FR_S$.
It follows that $\X-\Y$ is in the set $E$ defined as in Lemma~\ref{lem2}. Hence,
\[
\normF{c_S(\A\ast\X-\A\ast\Y)}\ge\gamma_S\normF{\X-\Y}
\]
for some constant $\gamma_S>0$.  By the definition of the set $S$,
\[
\normF{c_S(\A\ast\X-\A\ast\Y)} = \normF{c_S(\A\ast\X-B)}.
\]
Since $\normF{c_S(\A\ast\X-B)}\le \normF{c_T(\A\ast\X-B)}$,
it follows that
\[
 \normF{c_T(\A\ast\X-B)}\ge\gamma_S\normF{\X-\Y}.
\]
Since there are only a finite number of possible $S$ given in (\ref{set:indices}),
\[
\min_S \gamma_S > 0
\]
Therefore, let $\gamma:=(\min_S \gamma_S)^{-1}>0,$
\[
d(\X,\FR)\le\gamma\normF{c_T(\A\ast\X-\B)},
\]
for all $\X\in\mathbb R^{l\times p\times n}$.
\end{proof}
\noindent
The smallest constant $\gamma>0$ such that the inequality (\ref{eq:hoffman}) holds for all $\X\in\mathbb R^{l\times p\times n}$ 
can be considered as a generalization of the Hoffman constant to our tensor setting.

It was shown in \cite[Lemma 2.8]{chen2021regularized} that
\[
\normF{\A\ast\X}\le\sqrt{n}\normF{\A}\normF{\X}=\normF{\fft{\A}}\normF{\X}.
\]
We use the DFT of $\A$ to derive a tighter upper bound in the following lemma.
\begin{lem}\label{lem:productNorm}
If $\mathcal A\in\mathbb R^{m\times l\times n}$ and $\mathcal X\in\mathbb R^{l\times p\times n}$, then
\[
\normF{\A\ast\X}\le\max_{1\le k\le n}\norm{\fft{\A}_k}\normF{\X}.
\]
\end{lem}
\begin{proof}
By (\ref{eq:bdiag}),
\begin{align*}
\normF{\A\ast\X}^2 &= \normF{\bcirc{\A}\unfold{\X}}^2\\
&= \normF{\bdiag(\fft{A})(F_n\otimes I_l)\unfold{\X}}^2\\
&=\sum_{k=1}^{n}\normF{\fft{\A}_k(F_n(k,:)\otimes I_l)\unfold{\X}}^2\\
&\le \sum_{k=1}^{n}\norm{\fft{\A}_k}^2\normF{(F_n(k,:)\otimes I_l)\unfold{\X}}^2 ~(\textrm{by singular value decomposition})\\
&= \max_{1\le k\le n}\norm{\fft{\A}_k}^2\sum_{k=1}^{n}\normF{(F_n(k,:)\otimes I_l)\unfold{\X}}^2\\
& = \max_{1\le k\le n}\norm{\fft{\A}_k}^2\normF{\X}^2.
\end{align*}
The proof is thus completed.
\end{proof}

Now we show that Algorithm~\ref{alg:trklc} converges linearly in expectation,
with the rate of convergence given in terms of the generalized Hoffman constant.
\begin{thm}\label{thm:tRKconverge}
Suppose the system (\ref{eq:LC}) has a nonempty feasible region $\FR$. Algorithm~\ref{alg:trklc} generates a sequence of iterations $\Set{\X^k}_{k\ge 0}$ 
that converges linearly in expectation:
\begin{equation}\label{eq:tRKconverge}
\expect [d(\X^\kp1,\FR)^2 | \X^k] \le\left(1-\min_{1\le i \le m}2t_i\left(1-\frac{t_{i}\max_{1\le j \le n}\normF{\fft{\A_{i::}}_j}^2}{2\normF{\A_{i::}}^2}\right) \frac{1}{\gamma^2\normF{\A}^2}\right)d(\X^k,\FR)^2
\end{equation}
for each $k$, where $\fft{\cdot}$ is the DFT given by (\ref{eq:fft}),
and $\gamma>0$ is the smallest constant such that (\ref{eq:hoffman}) holds. 
\end{thm}
\begin{proof}
Let $\projF{\cdot}$ denote the projection onto the feasible region $\FR$. If $i_k\in I_\le$, then, by Lemma~\ref{lem:productNorm},
\begin{align*}
\normF{\X^\kp1-\projF{\X^\kp1}}^2  &\le \normF{\X^\kp1-\projF{\X^k}}^2\\
& \le \left\|\X^k-\projF{\X^k}-t_{i_k}\frac{\A_{i_k::}^T \ast ( \A_{i_k::}\ast\X^k - \B_{i_k::} )_+}{\|\A_{i_k::}\|_F^2}\right\|_F^2\\
& \le \normF{\X^k-\projF{\X^k}}^2+t_{i_k}^2 \frac{\normF{\A_{i_k::}^T\ast (\A_{i_k::}\ast\X^k-\B_{i_k::})_+}^2}{\normF{\A_{i_k::}}^4}\\
&\quad -\frac{2t_{i_k}}{\normF{\A_{i_k::}}^2}\left<\X^k-\projF{\X^k},~\A_{i_k::}^T\ast ( \A_{i_k::}\ast\X^k - \B_{i_k::})_+\right>\\
&\le  \normF{\X^k-\projF{\X^k}}^2\\
&\quad -2t_{i_k}(1-\frac{t_{i_k}}{2\normF{\A_{i_k::}}^2}\max_{1\le j \le n}\norm{\fft{\A_{i_k::}}_j}^2) 
\frac{\normF{(\A_{i_k::}\ast\X^k-\B_{i_k::})_+}^2}{\normF{\A_{i_k::}}^2}\\
&=  \normF{\X^k-\projF{\X^k}}^2\\
&\quad -2t_{i_k}(1-\frac{t_{i_k}}{2\normF{\A_{i_k::}}^2}\max_{1\le j \le n}\normF{\fft{\A_{i_k::}}_j}^2) 
\frac{\normF{(\A_{i_k::}\ast\X^k-\B_{i_k::})_+}^2}{\normF{\A_{i_k::}}^2},
\end{align*}
where $\fft{\A_{i_k::}}_j$ is the $j$-th frontal slice of $\fft{\A_{i_k::}}$.

Similarly, if $i_k\in I_=$, then
\[\hspace{-1em}
\normF{\X^\kp1-\projF{\X^\kp1}}^2 \le \normF{\X^k-\projF{\X^k}}^2- 2t_{i_k}(1-\frac{t_{i_k}}{2\normF{\A_{i_k::}}^2}\max_{1\le j \le n}\normF{\fft{\A_{i_k::}}_j}^2) \frac{\normF{\A_{i_k::}\X^k-\B_{i_k::}}^2}{\normF{\A_{i_k::}}^2}.
\]
Taking the expectation with respect to the specified probability distribution gives
\[
\expect [d(\X^\kp1,\FR)^2 | \X^k] \le d(\X^k,\FR)^2-\min_{1\le i \le m}2t_i(1-\frac{t_{i}}{2\normF{\A_{i::}}^2}\max_{1\le j \le n}\normF{\fft{\A_{i::}}_j}^2) \frac{\normF{c_T(\A\ast\X^k-\B)}^2}{\normF{\A}^2}.
\]
It follows from Theorem~\ref{thm:dist} that, there exists a constant $\gamma>0$ such that
\[
\expect [d(\X^\kp1,\FR)^2 | \X^k] \le \left(1-\min_{1\le i \le m}2t_i\left(1-\frac{t_{i}\max_{1\le j \le n}\normF{\fft{\A_{i::}}_j}^2}{2\normF{\A_{i::}}^2}\right) \frac{1}{\gamma^2\normF{\A}^2}\right)d(\X^k,\FR)^2.
\]
\end{proof}

{\bf Implications for convergence.} The above theorem indicates that the TRK-L method converges linearly in expectation if the step sizes $\Set{t_i}$ are subject to the following upper bounds:
\[
t_i < 2\normF{\A_{i::}}^2/(\max_{1\le j\le n}{\normF{\fft{A_{i::}}_j}^2})
\]
for each $i\in\Set{1,\cdots,m}.$ To take a closer look at these upper bounds, we note that, by (\ref{eq:Fnorm}),
\[
\normF{\A_{i::}}^2 = \frac{1}{n}\normF{\fft{\A_{i::}}}^2 = \frac{1}{n}\sum_{j=1}^n\normF{(\fft{\A_{i::}})_j}^2.
\]
Therefore, the upper bounds on $\Set{t_i}$ satisfy 
\[
\frac{2}{n}\le2\normF{\A_{i::}}^2/(\max_{1\le j\le n}{\normF{\fft{\A_{i::}}_j}^2})\le 2.
\]
The lower bound $\frac{2}{n}$ is achieved when $\A_{i::}$ has only one nonzero frontal slice, and the upper bound $2$ is achieved if
the norms of the frontal slices $\normF{\fft{A_{i::}}_j}$ are constant for all $j$.

\section{A comparison of TRK-L and B-MRK}\label{sec:comp}
In this section, we draw connections between the TRK-L method proposed in Section~\ref{sec:trkL} and the B-MRK method proposed in Section~\ref{sec:b-mrk}.
By applying the $\unfold{\cdot}$ operator on both sides of system (\ref{eq:LC}), the tensor system can be equivalently rewritten in the matrix form:
\begin{equation}\label{eq:bcircLC}
\begin{cases}
\bcirc{\A_{i::}}\unfold{\X}\le\unfold{\B_{i::}}& (i \in I_\le) \\
\bcirc{\A_{i::}}\unfold{\X}=\unfold{\B_{i::}}& (i \in I_=).
\end{cases}
\end{equation}
The updates in Algorithm~\ref{alg:trklc} may also be rewritten in matrix forms:
\begin{equation}\label{eq:update1}
\unfold{\X^\kp1} = \unfold{\X^k}-\alpha_{i_k}\bcirc{\A_{i_k::}}^T
\unfold{(\A_{i_k::}\ast\X^k-\B_{i_k::})_+}
\end{equation}
for $i_k\in I_\le$, and
\begin{equation}\label{eq:update2}
\unfold{\X^\kp1} = \unfold{\X^k}-\alpha_{i_k}\bcirc{\A_{i_k::}}^T
\unfold{\A_{i_k::}\ast\X^k-\B_{i_k::}}
\end{equation}
for $i_k\in I_=$,
with 
\begin{equation}\label{eq:alpha1}
\alpha_{i_k}=\frac{t_{i_k}}{\normF{\unfold{\A_{i_k::}}}^2}=\frac{t_{i_k}}{\normF{\A_{i_k::}}^2}<\frac{2}{(\max_{1\le j\le n}{\normF{\fft{\A_{i_k::}}_j}^2})}.
\end{equation}

Now for each $i\in\Set{1,\cdots,m}$, define $\tau_i$ to be the set of row indices associated with the rows of $\bcirc{\A_{i::}}$ in $\bcirc{\A}$.
Applying Algorithm~\ref{alg:block_mrk} to system (\ref{eq:bcircLC}), a random block $\tau =\tau_{i_k}$ is selected at each iteration $k$, and the updating rules are given by
\begin{equation}\label{eq:update3}
\unfold{\X^\kp1} = \unfold{\X^k}-\tilde\alpha_{i_k}\bcirc{\A_{i_k::}}^T
\unfold{(\A_{i_k::}\ast\X^k-\B_{i_k::})_+}
\end{equation}
for $i_k\in I_\le$, and
\begin{equation}\label{eq:update4}
\unfold{\X^\kp1} = \unfold{\X^k}-\tilde\alpha_{i_k}\bcirc{\A_{i_k::}}^T
\unfold{\A_{i_k::}\ast\X^k-\B_{i_k::}}
\end{equation}
for $i_k\in I_=$,
with 
\begin{equation}\label{eq:alpha2}
\tilde\alpha_{i_k} = \frac{t_{\tau}}{{\normF{\bcirc{\A_{i_k::}}}^2}}=\frac{t_{\tau}}{{n^4\normF{\A_{i_k::}}^2}}<\frac{2}{{n^4\normF{\A_{i_k::}}^2}}.
\end{equation}
The updating rules (\ref{eq:update1})-(\ref{eq:update2}) of TRK-L in matrix forms and the updating rules (\ref{eq:update3})-(\ref{eq:update4}) of B-MRK differ 
by the coefficients $\alpha_{i_k}$ and $\tilde\alpha_{i_k}$.
The upper bound on $\alpha_{i_k}$ given by (\ref{eq:alpha1}) satisfy that 
\[
\frac{2}{(\max_{1\le j\le n}{\normF{\fft{\A_{i_k::}}_j}^2})}\ge\frac{2}{\normF{\fft{\A_{i_k::}}}^2}=\frac{2}{n\normF{\A_{i_k::}}^2}.
\]
Therefore, $\alpha_{i_k}$ has a looser upper bound than  $\tilde\alpha_{i_k}$ for all $k.$

We now compare the two methods' rates of convergence. For TRK-L, the rate of convergence is given in Theorem~\ref{thm:tRKconverge}. By choosing the step size 
$t_i=\normF{\A_{i::}}^2/(\max_{1\le j\le n}{\normF{\fft{\A_{i::}}_j}^2})$ for each $i$ in (\ref{eq:tRKconverge}), we obtain
\begin{align}\label{eq:tub}
\expect [d(\X^\kp1,\FR)^2 | \X^k] \le & \left(1-\min_{1\le i\le m}\frac{\normF{\A_{i::}}^2}{\max_{1\le j \le n}\normF{\fft{\A_{i::}}_j}^2}\frac{1}{\gamma^2\normF{\A}^2}\right)d(\X^k,\FR)^2\nonumber\\ 
\le & \left(1-\frac{1}{n\gamma^2\normF{\A}^2}\right)d(\X^k,\FR)^2,
\end{align}
where $\gamma$ is the generalized Hoffman constant given by (\ref{eq:hoffman}).
For B-MRK, the rate of convergence is shown in Theorem~\ref{thm:tRKconverge}. Applying B-MRK to the system (\ref{eq:bcircLC}), 
the theoretical upper bound of the expected error in (\ref{eq:mRKconverge}) satisfies
\begin{align}\label{eq:mub}
\Bigg(1-\min_{\tau} & 2 t_\tau \left(1-\frac{t_\tau}{2}\right) \frac{1}{L^2\normF{\bcirc{\A}}^2}\Bigg) d(\unfold{\X^k},\FR')^2\nonumber\\
&\ge\left(1-\frac{1}{L^2\normF{\bcirc{\A}}^2}\right)d(\X^k,\FR)^2
=\left(1-\frac{1}{n^4L^2\normF{\A}^2}\right)d(\X^k,\FR)^2,
\end{align}
where $\FR'$ is the feasible region defined by (\ref{eq:bcircLC}). Since the tensor system (\ref{eq:LC}) is equivalent to the matrix system (\ref{eq:bcircLC}),
the Hoffman constant $L$ for (\ref{eq:bcircLC}) should be identical to the generalized Hoffman constant $\gamma$ for (\ref{eq:LC}).
Hence, the theoretical upper bound of the expected error of TRK-L derived in (\ref{eq:tub})
is smaller than that of B-MRK given in (\ref{eq:mub}). This implies that, with suitable choices of the step sizes, TRK-L provides a tighter theoretical guarantee than that of B-MRK. 
We will further compare TRK-L and B-MRK by numerical experiments in Section~\ref{sec:exp}.

\section{A TRK method for Linear Equality and Bound Constraints}\label{sec:trklb}
In this section, we consider problems of the form
\begin{equation}\label{eq:nC}
\A\ast\X = \B, ~\X \le \tilde\B,\\
\end{equation}
where $\A\in \mathbb{R}^{m\times l \times n}$, $\B \in \mathbb{R}^{m \times p \times n}$ and 
$\tilde\B\in\mathbb{R}^{l \times p \times n}$.
The system (\ref{eq:nC}) is a special case of the general form given by (\ref{eq:LC}). 
Therefore, it can be solved using the TRK-L method proposed previously in Section~\ref{sec:trkL}. 
However, considering that the orthogonal projection of a point $\X$ onto the half space $\Set{\X\le \tilde\B}$
can be easily computed as $\min\{\X,\tilde\B\}$, we propose another variant of the TRK method, TRK-LB, specifically 
tailored for solving problems of form (\ref{eq:nC}). 

The TRK-LB method is summarized in Algorithm~\ref{alg:trk_lb}. The method solves (\ref{eq:nC}) in an alternating manner.  First, the method projects towards the solution set of the tensor equation given by the sampled row slice of $\A$, thus progressing towards the solution set of $\A \ast \X = \B$.  Then, the method orthogonally projects this iterate onto the set given by $\Set{\X\le \tilde\B}$, 
ensuring this newly updated iterate satisfies the bound constraint.

\begin{algorithm}[!htp]
\caption{TRK-LB for tensor linear equality and bound constraints}\label{alg:trknc}
\begin{algorithmic}[1]
\State \textbf{Input:} \begin{itemize}
\item $\mathcal X^0 \in \mathbb{R}^{l\times p \times n}$, $\mathcal A\in \mathbb{R}^{m\times l \times n}$, $\mathcal{B} \in \mathbb{R}^{m \times p \times n}$, $\tilde\B \in \mathbb{R}^{l\times p \times n}$, 
\item probabilities $p_i = \|\mathcal A_{i::}\|_F^2/\|\mathcal A\|_F^2$ and 
\item step sizes 
 $t_i<2\normF{\A_{i::}}^2/(\max_{1\le j\le n}{\normF{\fft{A_{i::}}_j}^2})$ for $i=1,2,\cdots,m$
 \end{itemize}
\For{$k = 0,1,2,\dots$}
    \State Sample $i_k \sim \Set{p_i}$
    \State $\Z^k = \mathcal{X}^k - t_{i_k}\mathcal{A}_{i_k::}^T \ast \left( \mathcal{A}_{i_k::}\ast\mathcal{X}^k - \mathcal{B}_{i_k::} \right)
    /\|\mathcal A_{i_k::}\|_F^2$
    \State $\X^{k+1} = \min\{\Z^k, \tilde\B\}$
\EndFor
\State \textbf{Output:} last iterate $\mathcal{X}^{k+1}$
\end{algorithmic}
\label{alg:trk_lb}
\end{algorithm}

We show that TRK-LB converges linearly in expectation to the feasible region $\FR.$
\begin{thm}\label{thm:trkLBconverge}
Suppose the system (\ref{eq:nC}) has a nonempty feasible region $\FR$. Algorithm~\ref{alg:trknc} generates a sequence of iterations $\Set{\X^k}_{k\ge 0}$ 
that converges linearly in expectation:
\[
\expect [d(\X^\kp1,\FR)^2 | \X^k] \le \left(1-\min_{1\le i \le m}2t_i\left(1-\frac{t_{i}\max_{1\le j \le n}\normF{\fft{\A_{i::}}_j}^2}{2\normF{\A_{i::}}^2}\right) \frac{1}{\gamma^2\normF{\A}^2}\right)d(\X^k,\FR)^2
\]
for each $k$, where $\fft{\cdot}$ is the DFT given by (\ref{eq:fft}),
and $\gamma>0$ is the smallest constant such that (\ref{eq:hoffman}) holds. 
\end{thm}
\begin{proof}
Let $\projF{\cdot}$ denote the projection onto the feasible region $\FR$. Then
\begin{align*}
\normF{\Z^k-\projF{\Z^k}}^2  &\le \normF{\Z^k-\projF{\X^k}}^2\\
& \le \left\|\Z^k-\projF{\X^k}-t_{i_k}\frac{\A_{i_k::}^T \ast ( \A_{i_k::}\ast\X^k - \B_{i_k::} )}{\|\A_{i_k::}\|_F^2}\right\|_F^2\\
& \le \normF{\X^k-\projF{\X^k}}^2+t_{i_k}^2 \frac{\normF{\A_{i_k::}^T\ast (\A_{i_k::}\ast\X^k-\B_{i_k::})}^2}{\normF{\A_{i_k::}}^4}\\
&\quad -\frac{2t_{i_k}}{\normF{\A_{i_k::}}^2}\left<\X^k-\projF{\X^k},~\A_{i_k::}^T\ast ( \A_{i_k::}\ast\X^k - \B_{i_k::})\right>\\
&\le  \normF{\X^k-\projF{\X^k}}^2\\
&\quad -2t_{i_k}(1-\frac{t_{i_k}}{2\normF{\A_{i_k::}}^2}\max_{1\le j \le n}\normF{\fft{\A_{i_k::}}_j}^2) 
\frac{\normF{(\A_{i_k::}\ast\X^k-\B_{i_k::})}^2}{\normF{\A_{i_k::}}^2},
\end{align*}
where $\fft{\A_{i_k::}}_j$ is the $j$-th frontal slice of $\fft{\A_{i_k::}}$.
Moreover,
\begin{align*}
\normF{\X^\kp1-\projF{\X^\kp1}}^2 &\le \normF{\X^\kp1-\projF{\Z^k}}^2\\
&\le \normF{\Z^k-(\Z^k-\tilde\B)_+ -\projF{\Z^k}}^2\\
& = \normF{\Z^k-\projF{\Z^k}}^2+\normF{(\Z^k-\tilde\B)_+}^2-2\left<(\Z^k-\tilde\B)_+, \Z^k-\projF{\Z^k}\right>\\
& \le \normF{\Z^k-\projF{\Z^k}}^2+\normF{(\Z^k-\tilde\B)_+}^2-2\left<(\Z^k-\tilde\B)_+, \Z^k-\tilde\B\right>\\
& = \normF{\Z^k-\projF{\Z^k}}^2-\normF{(\Z^k-\tilde\B)_+}^2\\
&\le \normF{\Z^k-\projF{\Z^k}}^2.
\end{align*}
Hence,
\begin{align*}
\normF{\X^\kp1-\projF{\X^\kp1}}^2 &\le  \normF{\X^k-\projF{\X^k}}^2\\
&\quad -2t_{i_k}(1-\frac{t_{i_k}}{2\normF{\A_{i_k::}}^2}\max_{1\le j \le n}\normF{\fft{\A_{i_k::}}_j}^2) 
\frac{\normF{(\A_{i_k::}\ast\X^k-\B_{i_k::})}^2}{\normF{\A_{i_k::}}^2}.
\end{align*}
Similarly as in the proof of Theorem~\ref{thm:tRKconverge}, we then take the conditional expectation and obtain that
\[
\expect [d(\X^\kp1,\FR)^2 | \X^k] \le \left(1-\min_{1\le i \le m}2t_i\left(1-\frac{t_{i}\max_{1\le j \le n}\normF{\fft{\A_{i::}}_j}^2}{2\normF{\A_{i::}}^2}\right) \frac{1}{\gamma^2\normF{\A}^2}\right)d(\X^k,\FR)^2.
\]
\end{proof}
    
\section{Experiments}\label{sec:exp}
In this section, we present numerical experiments on the B-MRK method in Algorithm~\ref{alg:block_mrk}, the TRK-L method in Algorithm~\ref{alg:trklc} and 
the TRK-LB method in Algorithm~\ref{alg:trk_lb}.
All the numerical experiments were done using Matlab R2021a on a MacBook Pro with an Apple M2 Max chip, 12 cores, and 32 GB of memory running macOS Ventura, Version 13.2.1.
\subsection{Experiments on the B-MRK method}
We test B-MRK on a randomly generated linear system of form (\ref{eq:mLC}), with $A\in\mathbb R^{1200\times100}$,
$X\in\mathbb R^{100\times 7}$ and $B\in\mathbb R^{1200\times 7}$. The system comprises $500$ equality constraints and $700$
inequality constraints. Entries of $A$ are drawn independently from a standard normal distribution; and to ensure the system has a nonempty feasibility region, we choose 
$B$ by computing $A X$ for a Gaussian random matrix $X$ and then perturbing the rows of $AX$ associated with inequality constraints by adding the absolute values of standard normal random variables to obtain $B$. For each $X$, the \emph{residual error} is defined by
\[
e_M:=\normF{c_M(AX-B)},
\]
where the function $c_M$ is defined in (\ref{eq:cMfunc}). With a fixed block size $|\tau|$, we partition $A$ into
$|\tau|\times n$ blocks consisting of consecutive rows. Figure~\ref{fig:bMRK_steps}  and Figure~\ref{fig:bMRK_sizes} show the results of B-MRK with 
varied block size $|\tau|$ and step size $t$. Although Theorem~\ref{thm:mRKconverge} states that B-MRK is guaranteed to converge when
the step size is less then $2$, we observe from Figure~\ref{fig:bMRK_steps} that a larger step size may lead to more rapid convergence in practice.
The B-MRK method with $|\tau|=1$ and $t=1$ recovers the classic RK method. Figure~\ref{fig:bMRK_sizes} indicates that, by choosing a larger block size 
and a larger step size, B-MRK outperforms the classic RK method.

\begin{figure}[!htbp]
  \centering
  \includegraphics[width=0.4\textwidth]{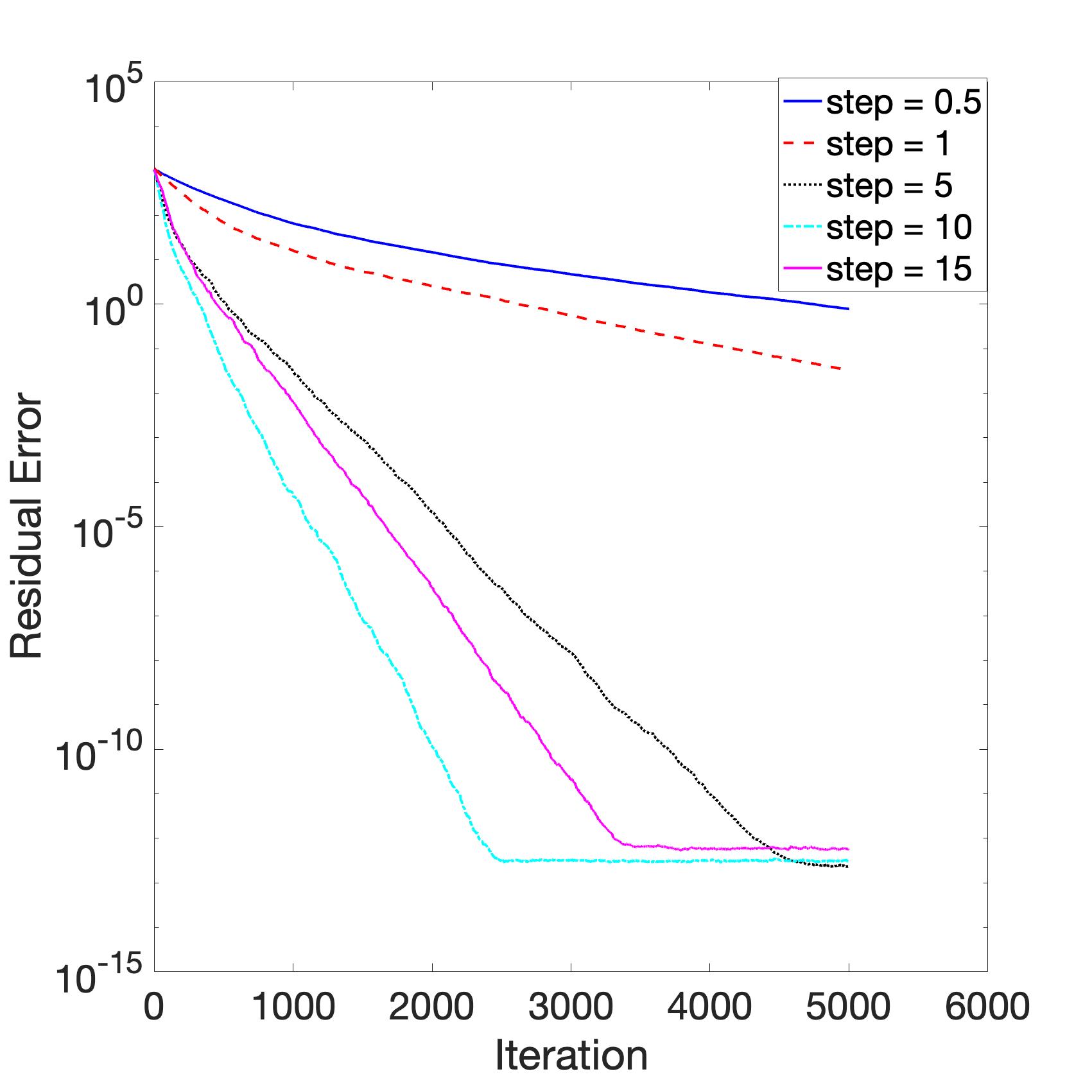}
  \hspace{2pt}
  \includegraphics[width=0.4\textwidth]{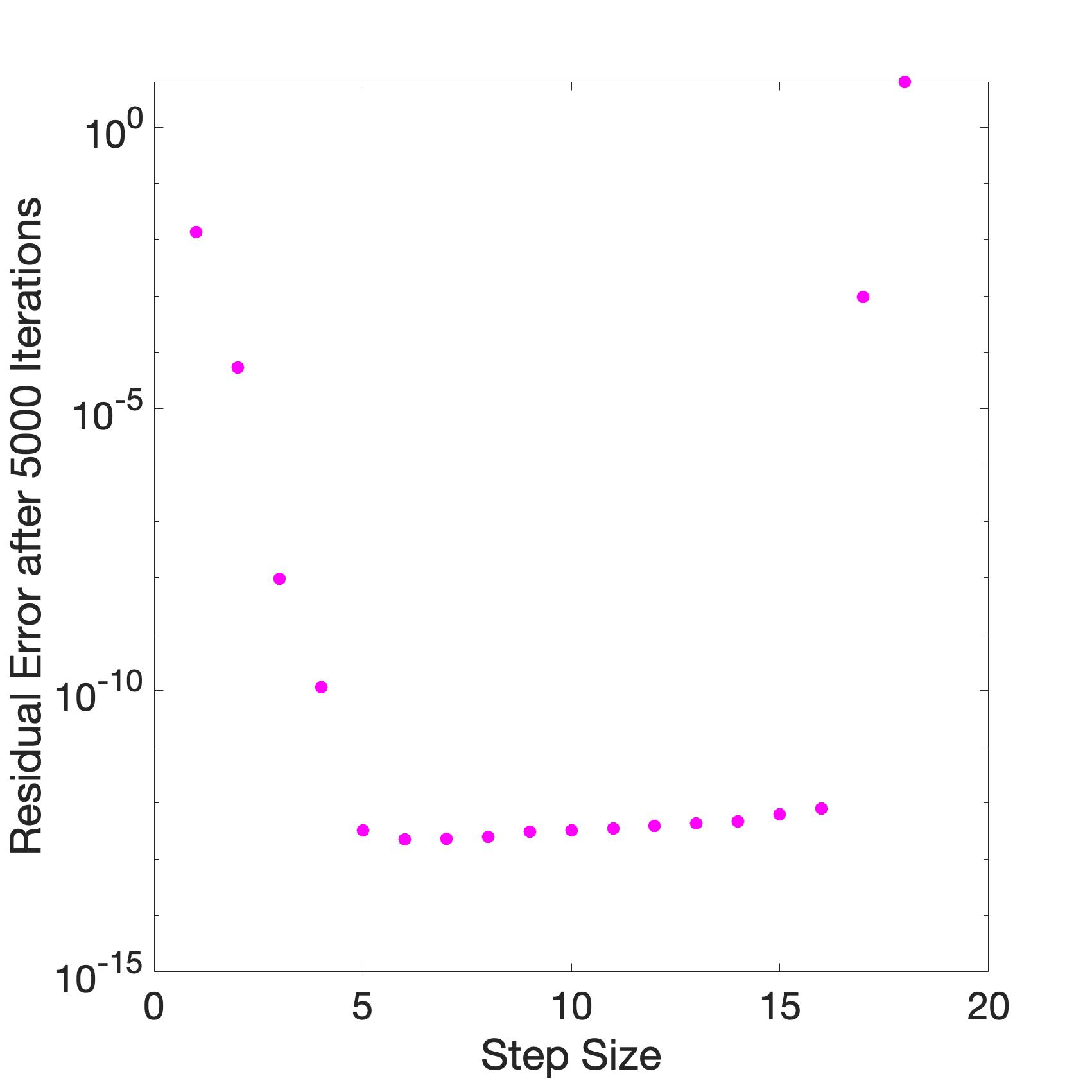}
  \caption{Results of B-MRK for a Gaussian random system with $A\in\mathbb R^{(500+700)\times100}$ and
  $B\in\mathbb R^{(500+700)\times 7}$.
  Left:  Iterations versus residual errors with block size $|\tau|=10$ and various step sizes.
  Right: Step size versus residual errors of B-MRK after 5000 iterations with $|\tau|=10$. }
  \label{fig:bMRK_steps}
\end{figure}

\begin{figure}[!htbp]
  \centering
  \includegraphics[width=0.4\textwidth]{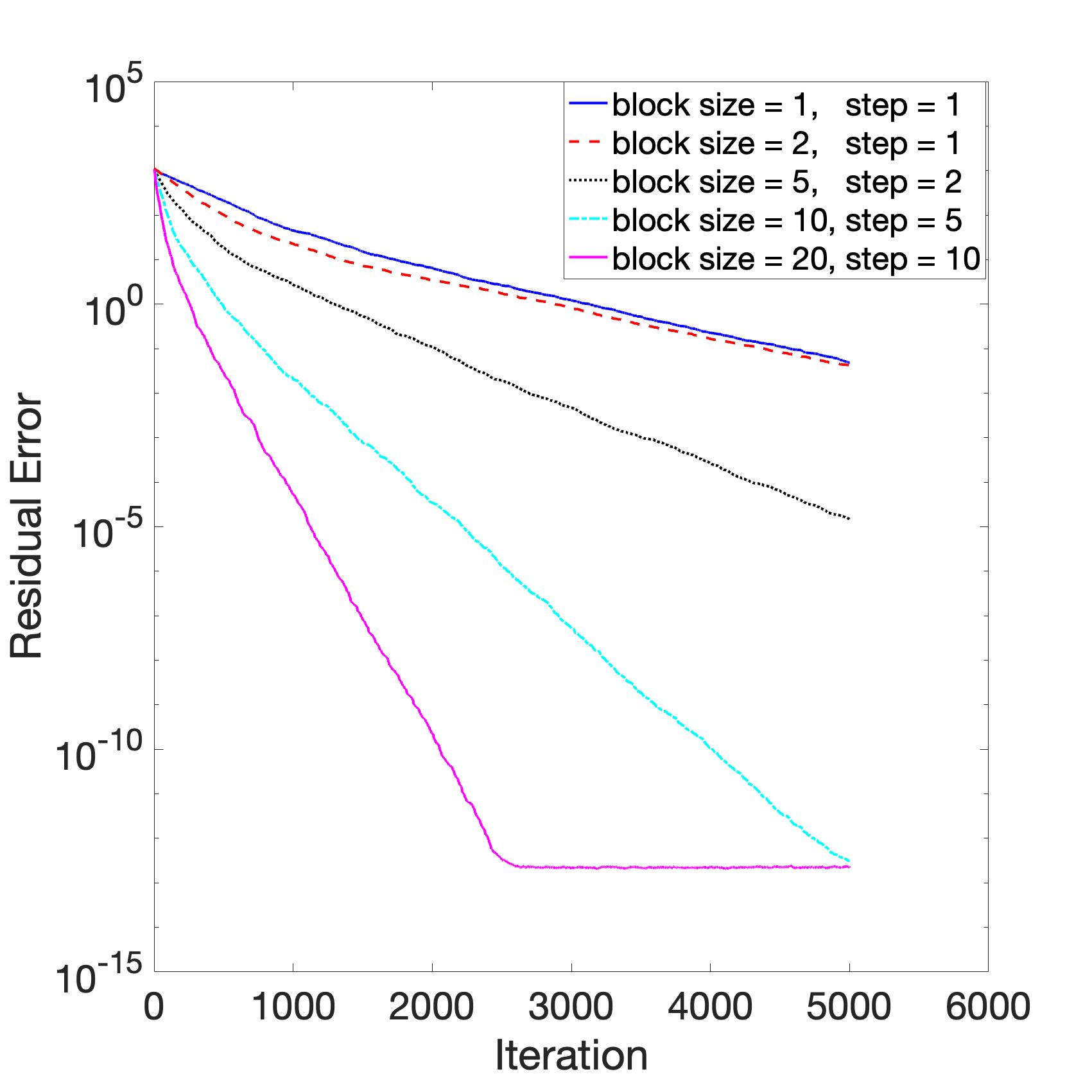}
  \hspace{2pt}
  \includegraphics[width=0.4\textwidth]{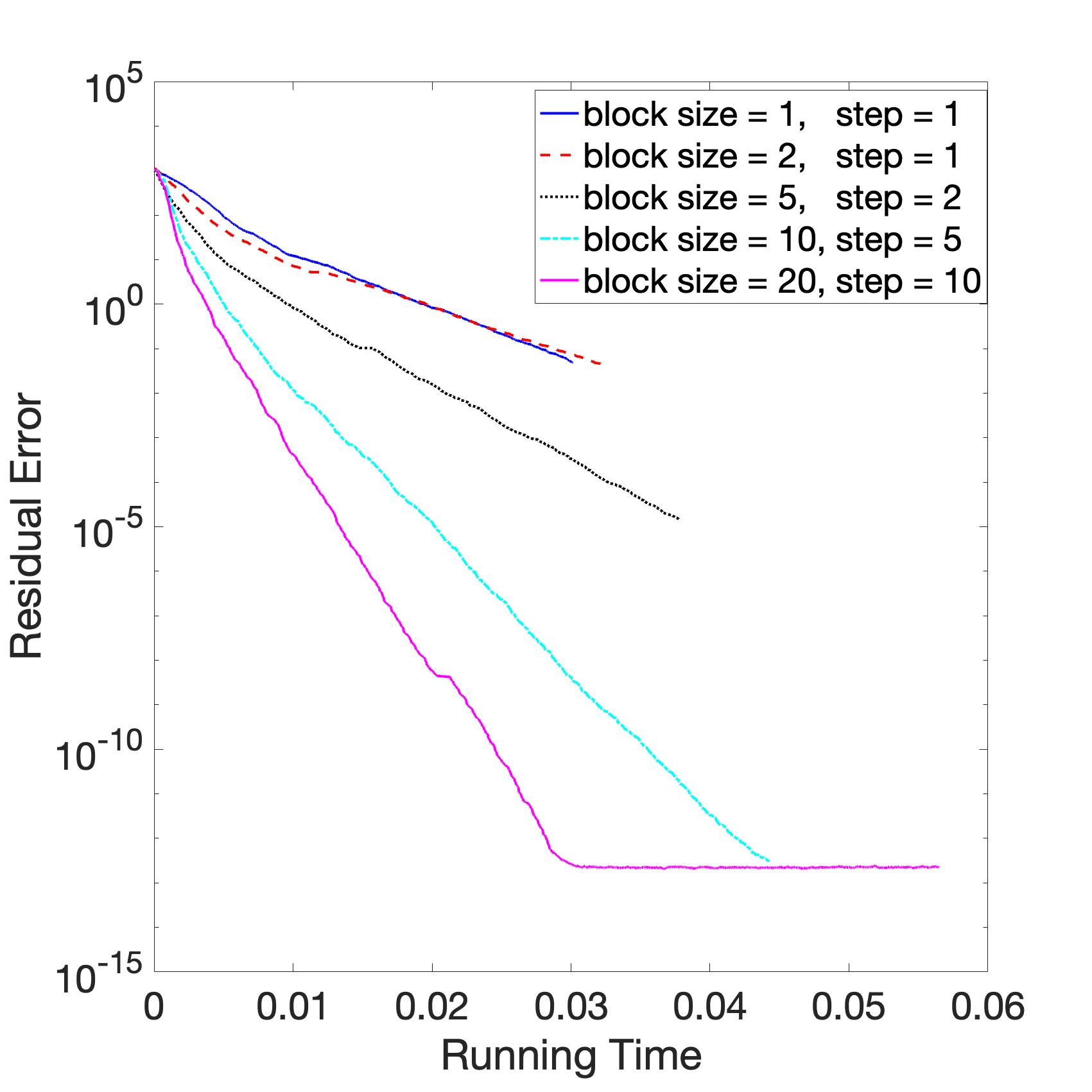}
  \caption{Results of B-MRK for a Gaussian random system with $A\in\mathbb R^{(500+700)\times100}$ and
  $B\in\mathbb R^{(500+700)\times 7}$.
  Left:  Iterations versus residual errors of B-MRK with various combinations of block sizes and step sizes.
  Right:  Running time versus residual errors of B-MRK with various combinations of block sizes and step sizes.}
  \label{fig:bMRK_sizes}
\end{figure}

Additionally, we consider the application of B-MRK in binary classification. 
Let $m$ be the number of data points and $n$ the number of features.
We generate a random data matrix $X_D\in\mathbb R^{m\times n}$ 
and a random weight vector $w^*\in\mathbb R^{n}$.
Entries of $X_D$ and $w^*$ are drawn independently from a standard normal distribution. For each data point $X_D(i,:)$, we
create a label \[
y_i := \textrm{sign}(X_D(i,:)w^*)\in\Set{-1,1}.
\] We then use B-MRK to solve the classification problem \[
Ax\le b,
\]
where $A:=-\textrm{diag}(y)X_D$ and $b:=-10^{-5}e,$ with $e$ being the vector of ones.
Results are shown in Figure~\ref{fig:bMRK_perceptron} for two cases: (i) $m=10000$ and $n=100$; (ii) $m=10000$ and $n=500$.
We observe that the classic RK method with block size $|\tau|=1$ and step size $t=1$ is more advantageous when there 
are a large number of features. However, with suitable choices of the block size and the step size,
the B-MRK method can achieve faster convergence than the classic RK method in both cases.

\begin{figure}[!htbp]
  \centering
  \includegraphics[width=0.4\textwidth]{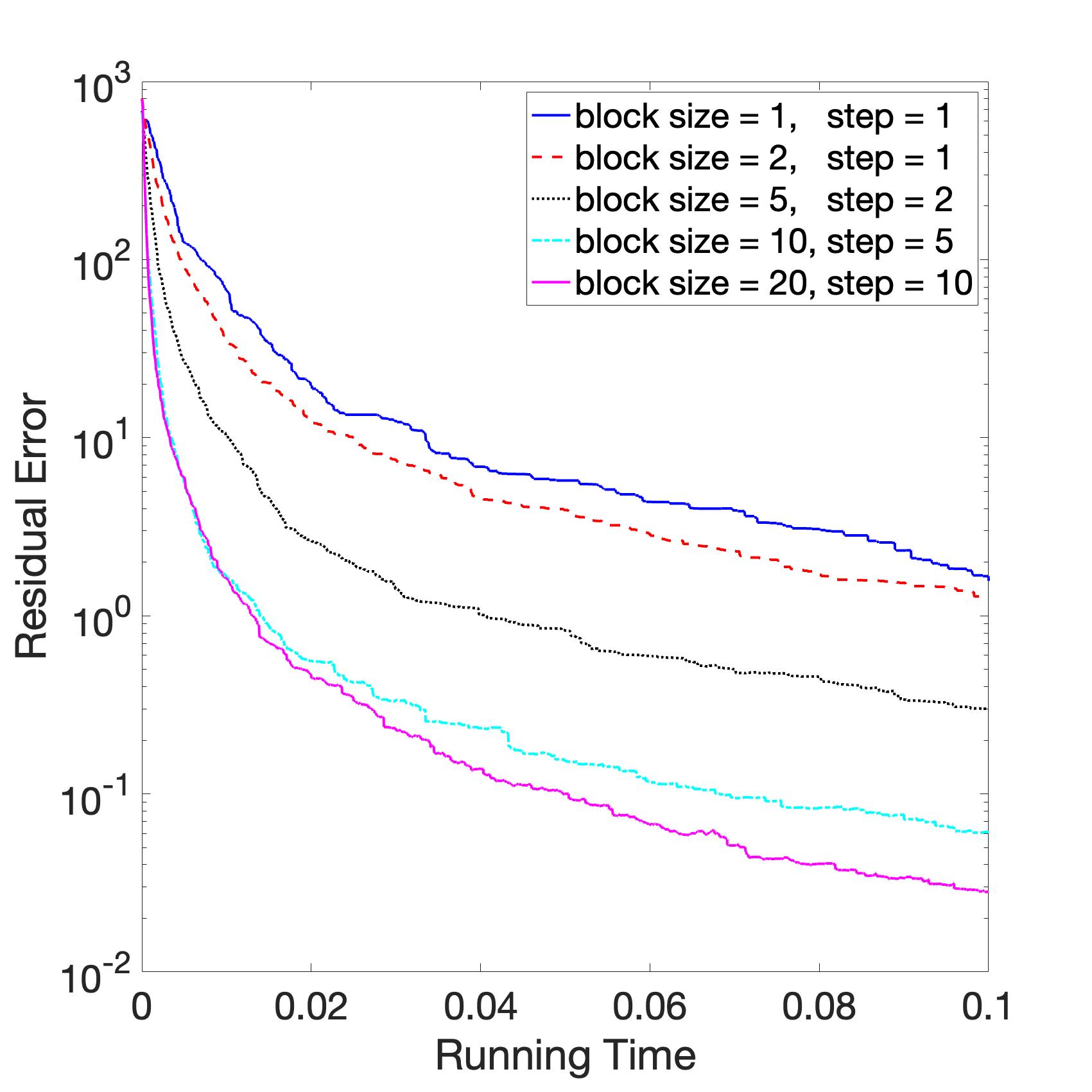}
  \hspace{2pt}
  \includegraphics[width=0.4\textwidth]{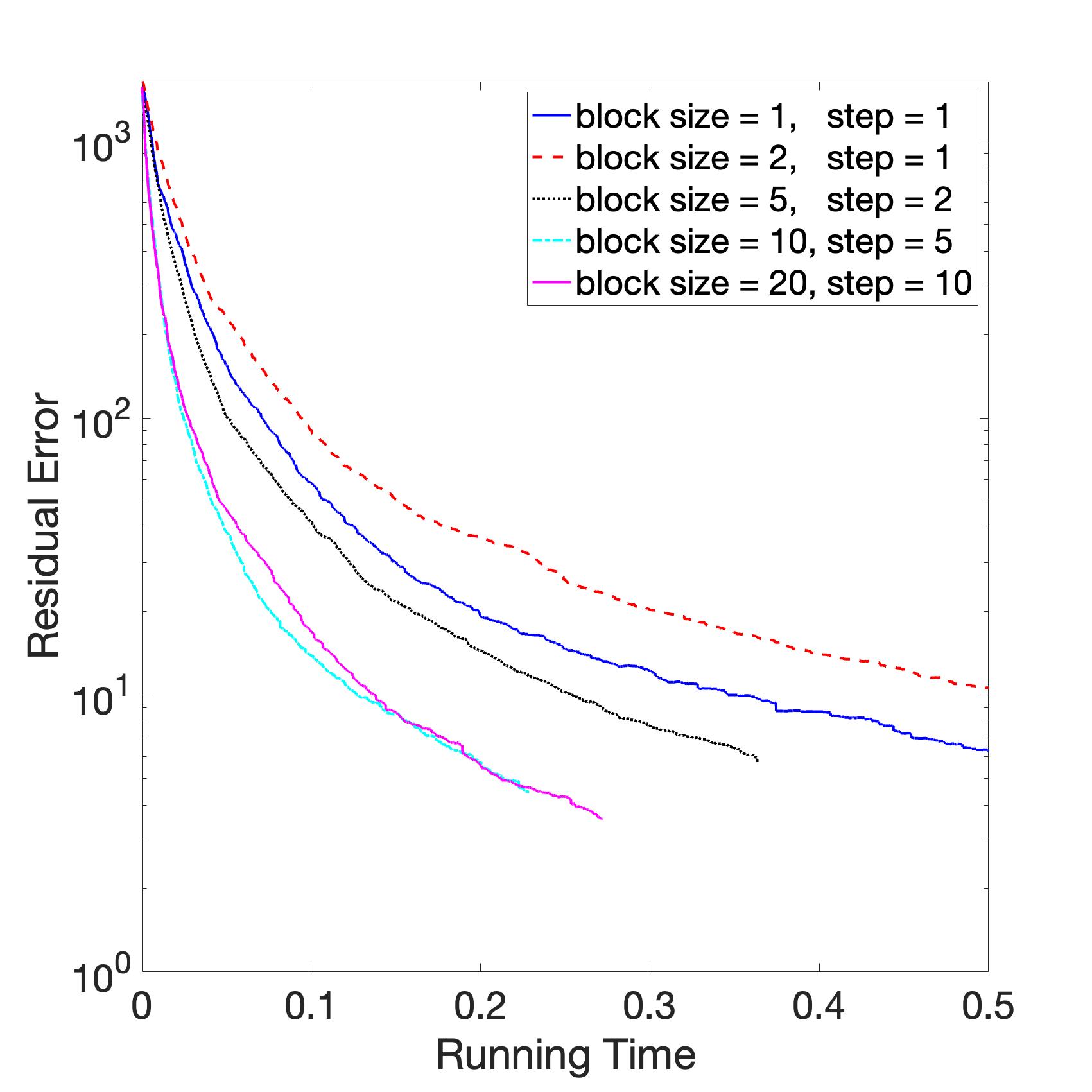}
  \caption{Results of B-MRK with various combinations of block sizes and step sizes for binary classification.
  Left:  Running time versus residual errors for $A\in\mathbb R^{10000\times 100}$.
  Right:  Running time versus residual errors for $A\in\mathbb R^{10000\times 500}$.}
  \label{fig:bMRK_perceptron}
\end{figure}

\subsection{Experiments on the TRK-L method}
We test TRK-L on a randomly generated tensor linear system of form (\ref{eq:LC}), with 
$\A\in\mathbb R^{120\times 50\times 10}$, $\X\in\mathbb R^{50\times 7\times 10}$ and $\mathcal B\in\mathbb R^{120\times 7\times 10}$.
The system comprises $50$ equality constraints and $70$
inequality constraints. Entries of $\A$ are drawn independently from a standard normal distribution; and to ensure the system has a nonempty feasibility region, we choose 
$\B$ by computing $\A \ast\X$ for a Gaussian random tensor $\X$ and then perturbing the row slices of $\A\ast\X$ associated with inequality constraints by adding the absolute values of standard normal random variables to obtain $\B$. For each $\X$, the \emph{residual error} is defined by
\[
e_T:=\normF{c_T(\A\ast\X-\B)},
\]
where the function $c_T$ is defined in (\ref{eq:cTfunc}). Theorem~\ref{thm:tRKconverge} implies that TRK-L is guaranteed to converge if, for each row
index $i,$ letting the step size
\[
t_i = \alpha\normF{\A_{i::}}^2/(\max_{1\le j\le n}{\normF{\fft{\A_{i::}}_j}^2})
\]
for a coefficient $\alpha<2.$
Figure~\ref{fig:tRKL_steps}  shows the results of TRK-L with various step coefficients $\alpha$.
In Figure~\ref{fig:tRKL_vs_bMRK}, we compare the results of applying TRK-L to the tensor system (\ref{eq:LC})
and applying B-MRK to its equivalent matrix system (\ref{eq:bcircLC}).
We observe that TRK-L outperforms B-MRK in this scenario.

\begin{figure}[!htbp]
  \centering
  \includegraphics[width=0.4\textwidth]{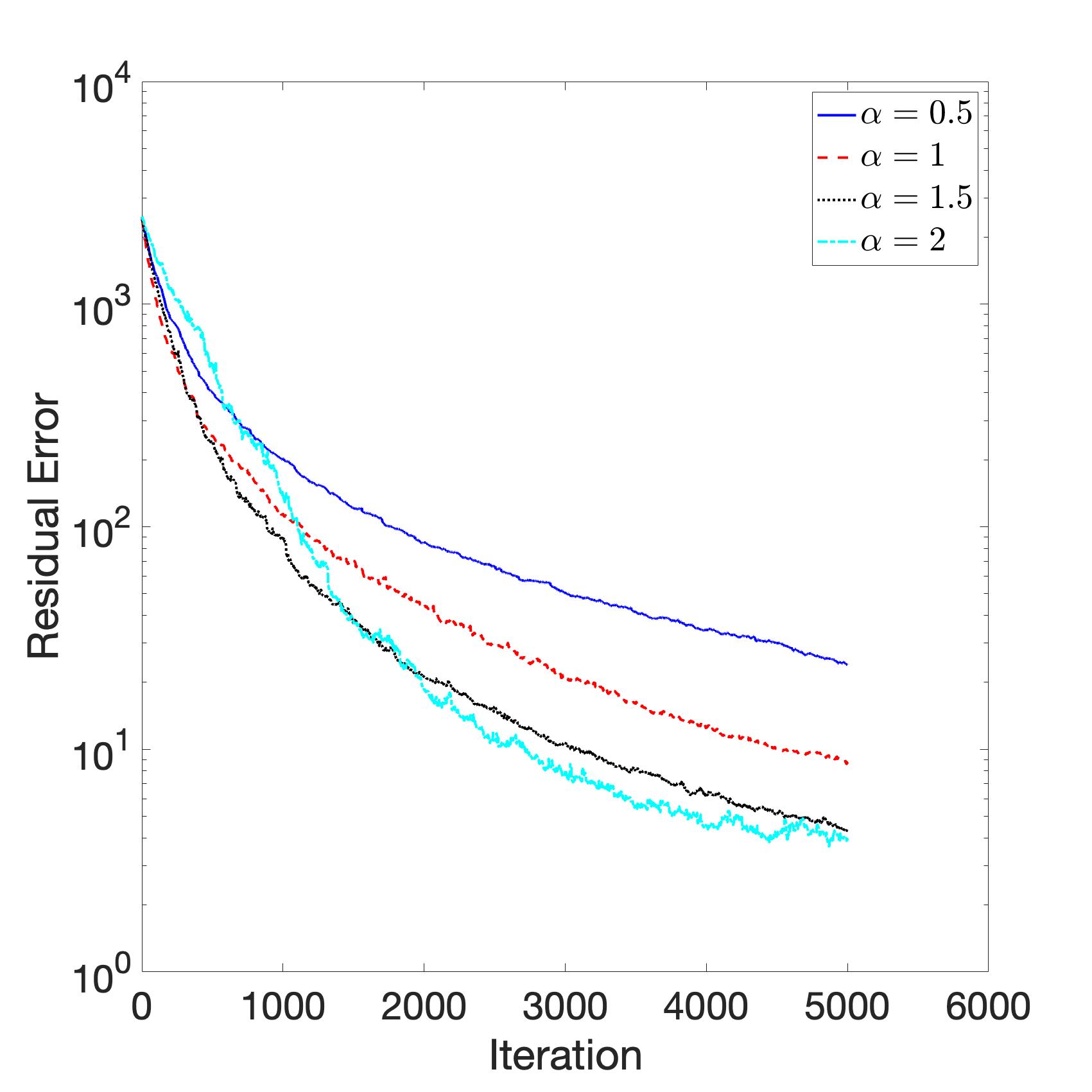}
  \hspace{2pt}
  \includegraphics[width=0.4\textwidth]{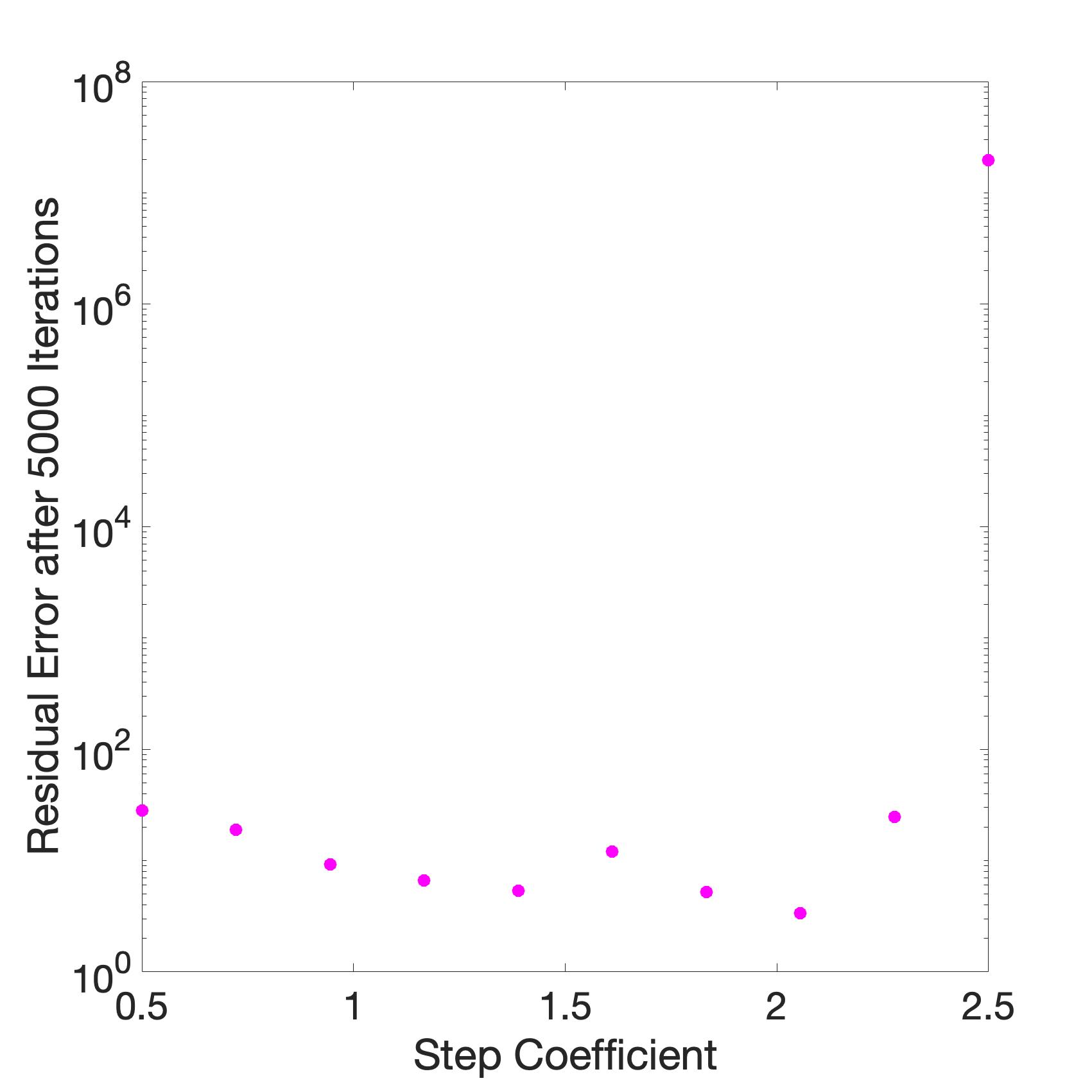}
  \caption{Results of TRK-L for a Gaussian random system with $\A\in\mathbb R^{(50+70)\times 50\times 10}$ and
  $\B\in\mathbb R^{(50+70)\times 7\times 10}$.
  Left:  Iterations versus residual errors with various step coefficients $\alpha$.
  Right: Step coefficient $\alpha$ versus residual errors after 5000 iterations.}
  \label{fig:tRKL_steps}
\end{figure}

\begin{figure}[!htbp]
  \centering
  \includegraphics[width=0.4\textwidth]{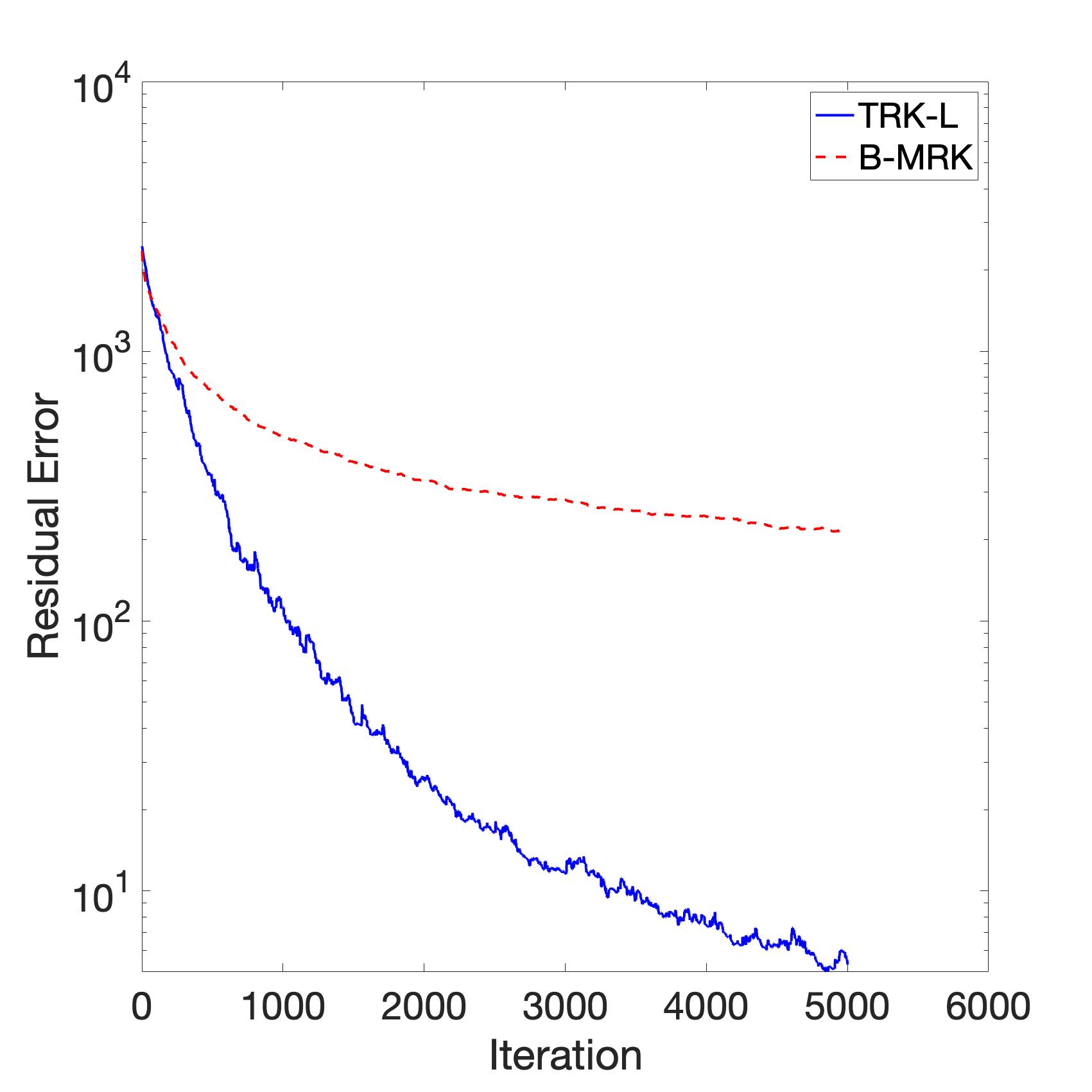}
  \hspace{2pt}
  \includegraphics[width=0.4\textwidth]{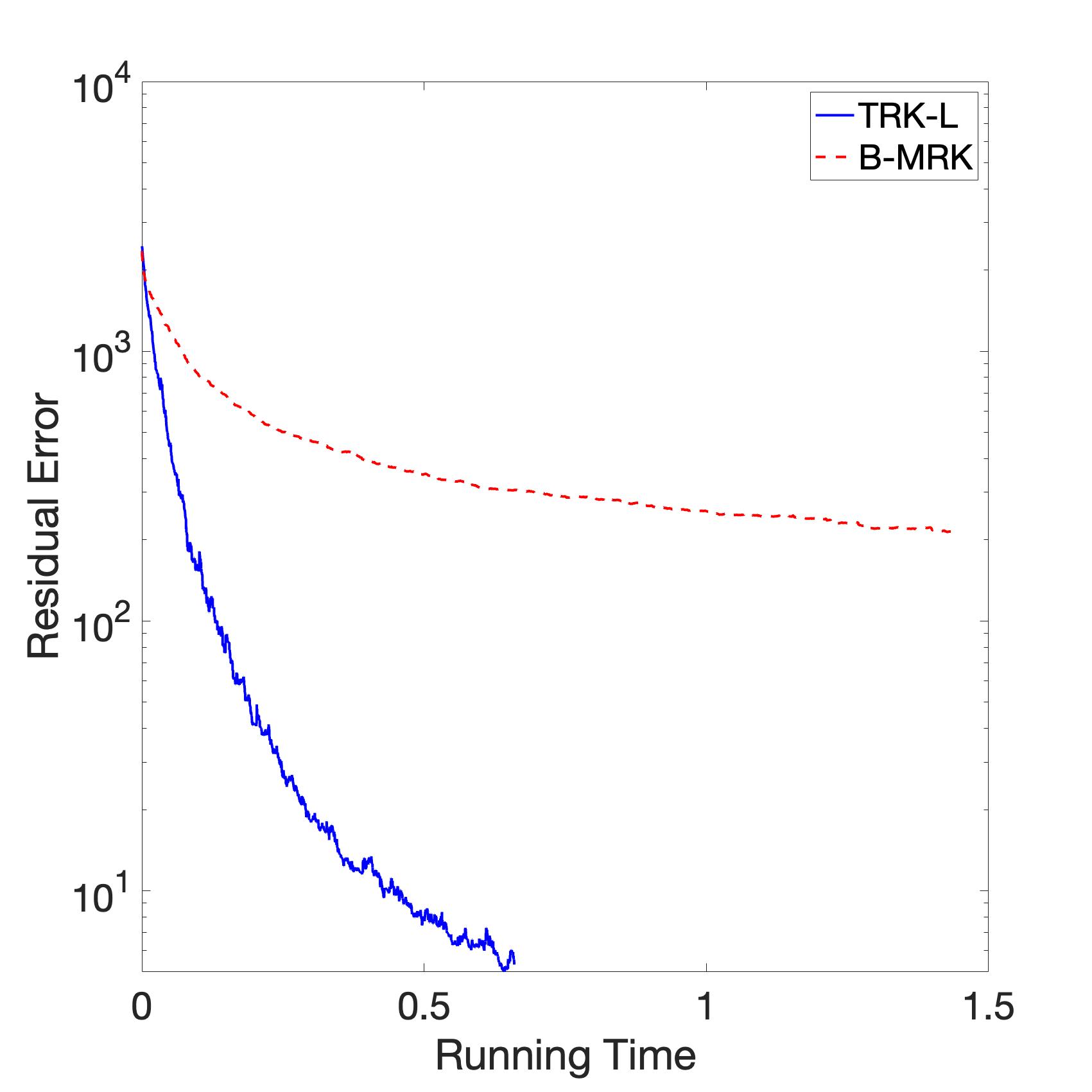}
  \caption{Comparison of TRK-L and B-MRK for a Gaussian random system with $\A\in\mathbb R^{(50+70)\times 50\times 10}$ and
  $\B\in\mathbb R^{(50+70)\times 7\times 10}$.
  Left:  Iterations versus residual errors of TRK-L with step coefficient $\alpha=1.8$ and B-MRK with step size $t=2$.
  Right:  Running time versus residual errors of TRK-L with step coefficients $\alpha=1.8$ and B-MRK with step size $t=2$.}
  \label{fig:tRKL_vs_bMRK}
\end{figure}

\subsection{Experiments on the TRK-LB method}
We test TRK-LB on a randomly generated tensor linear system of the special form (\ref{eq:nC}), with 
$\A\in\mathbb R^{100\times 50\times 10}$, $\mathcal B\in\mathbb R^{100\times 7\times 10}$, $\X\in\mathbb R^{50\times 7\times 10}$ 
and $\tilde\B\in\mathbb R^{50\times 7\times 10}$.
Entries of $\A$ are standard normal random variables; and to ensure the system has a nonempty feasibility region, we choose 
$\B$ to be $\A \ast\X$ for a Gaussian random tensor $\X$ and perturb the $\X$ by adding the absolute values of standard normal random variables to 
obtain $\tilde\B$. For each $\X$, the \emph{residual error} is defined by
\[
\tilde e_T:=\sqrt{\normF{\A\ast\X-\B}^2+\normF{(\X-\tilde\B)_+}^2},
\]
Similarly as TRK-L, Theroem~\ref{thm:trkLBconverge}  implies that TRK-LB is guaranteed to converge if letting
\[
t_i= \alpha\normF{\A_{i::}}^2/(\max_{1\le j\le n}{\normF{\fft{\A_{i::}}_j}^2}),\quad i=1,\cdots,m,
\]
for a coefficient $\alpha<2$.
Figure~\ref{fig:tRKLB_steps}  shows the results of TRK-LB with various step coefficients $\alpha$.
In Figure~\ref{fig:tRKLB_vs_tRKL}, 
We observe that TRK-LB converges more rapidly than TRK-L in this scenario.

\begin{figure}[!htbp]
  \centering
  \includegraphics[width=0.4\textwidth]{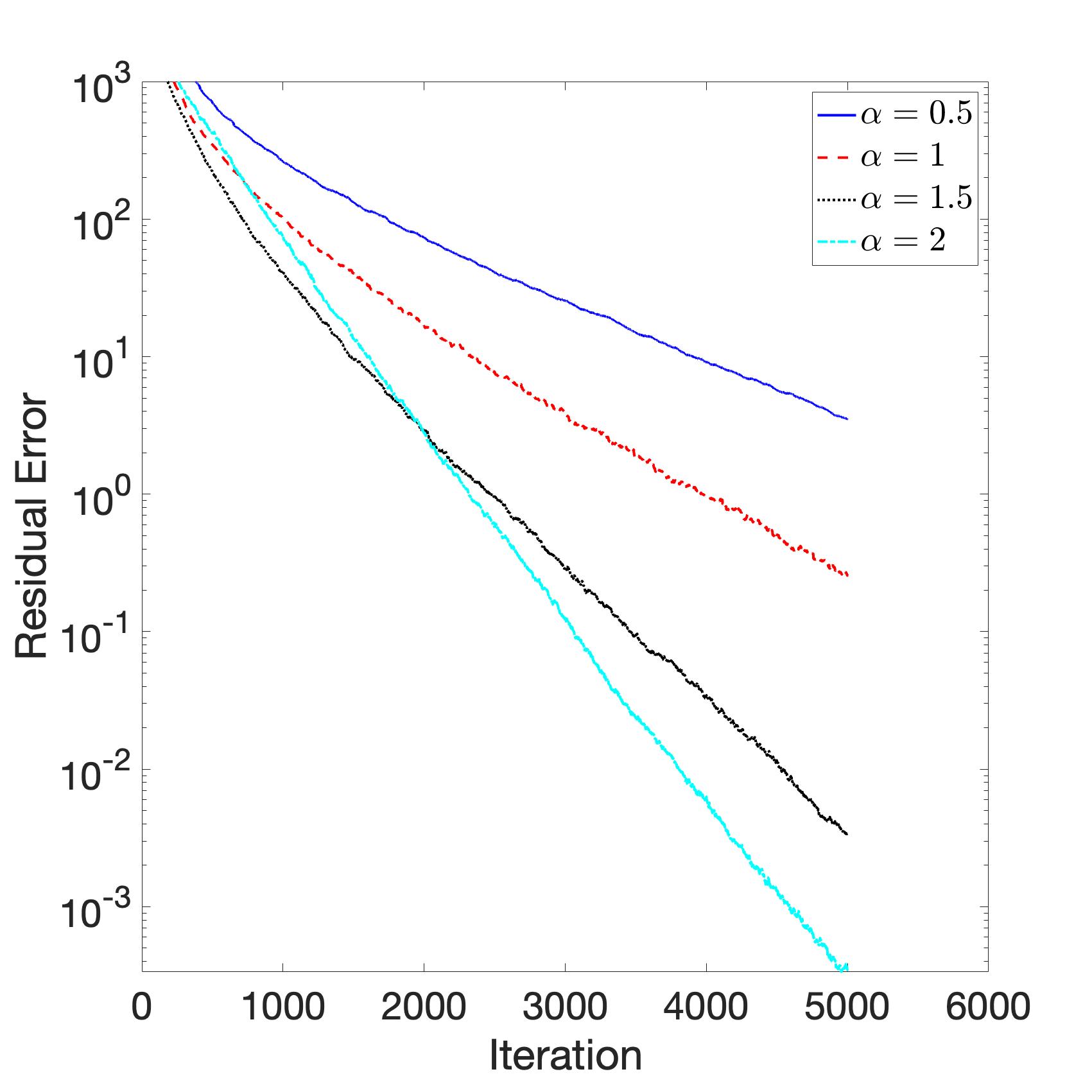}
  \hspace{2pt}
  \includegraphics[width=0.4\textwidth]{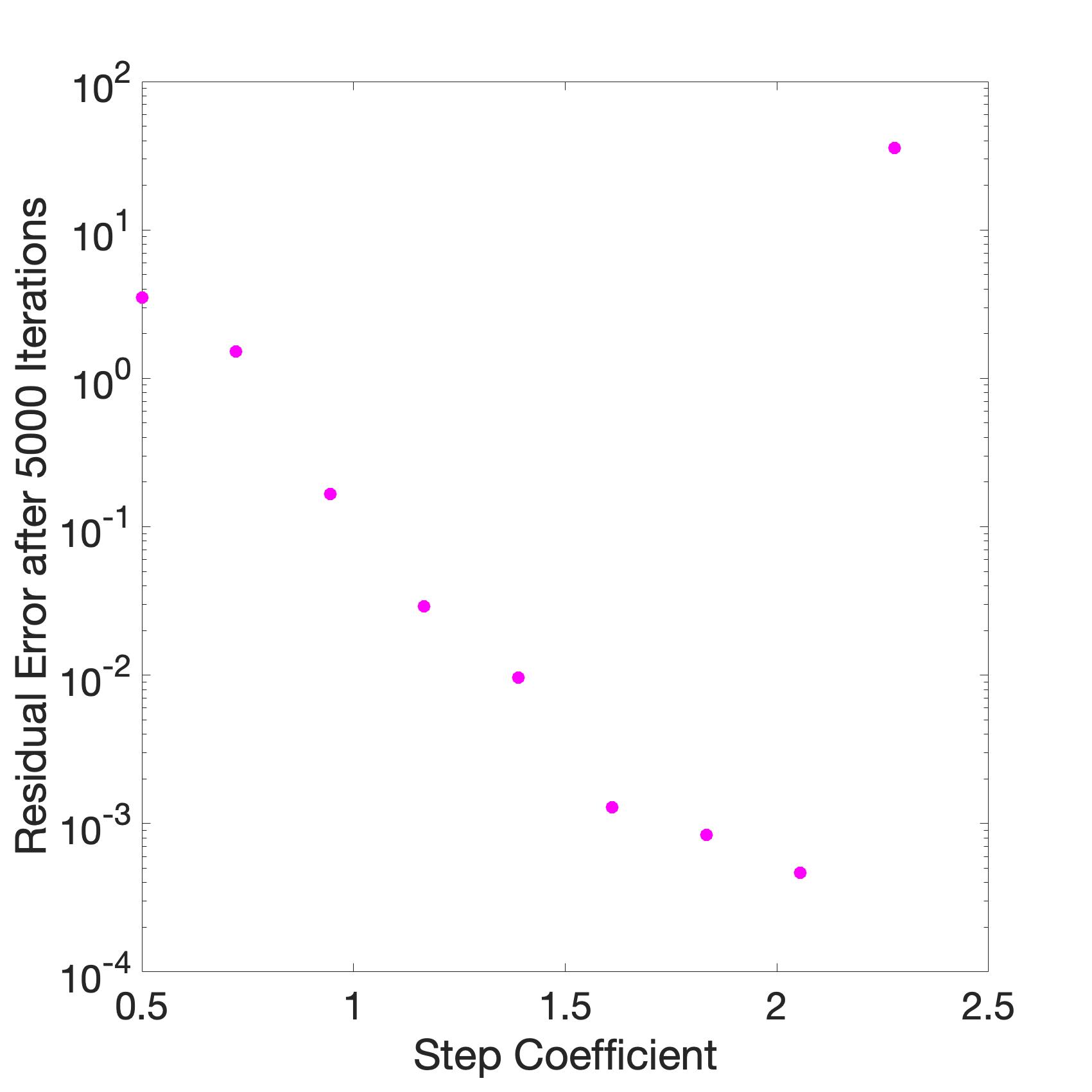}
  \caption{Results of TRK-LB for a Gaussian random system with $\A\in\mathbb R^{100\times 50\times 10}$,
  $\B\in\mathbb R^{100\times 7\times 10}$ and $\tilde \B\in\mathbb R^{50\times 7\times 10}$.
  Left:  Iterations versus residual errors with various step coefficients $\alpha$.
  Right: Step coefficient $\alpha$ versus residual errors after 5000 iterations.}
  \label{fig:tRKLB_steps}
\end{figure}

\begin{figure}[!htbp]
  \centering
  \includegraphics[width=0.4\textwidth]{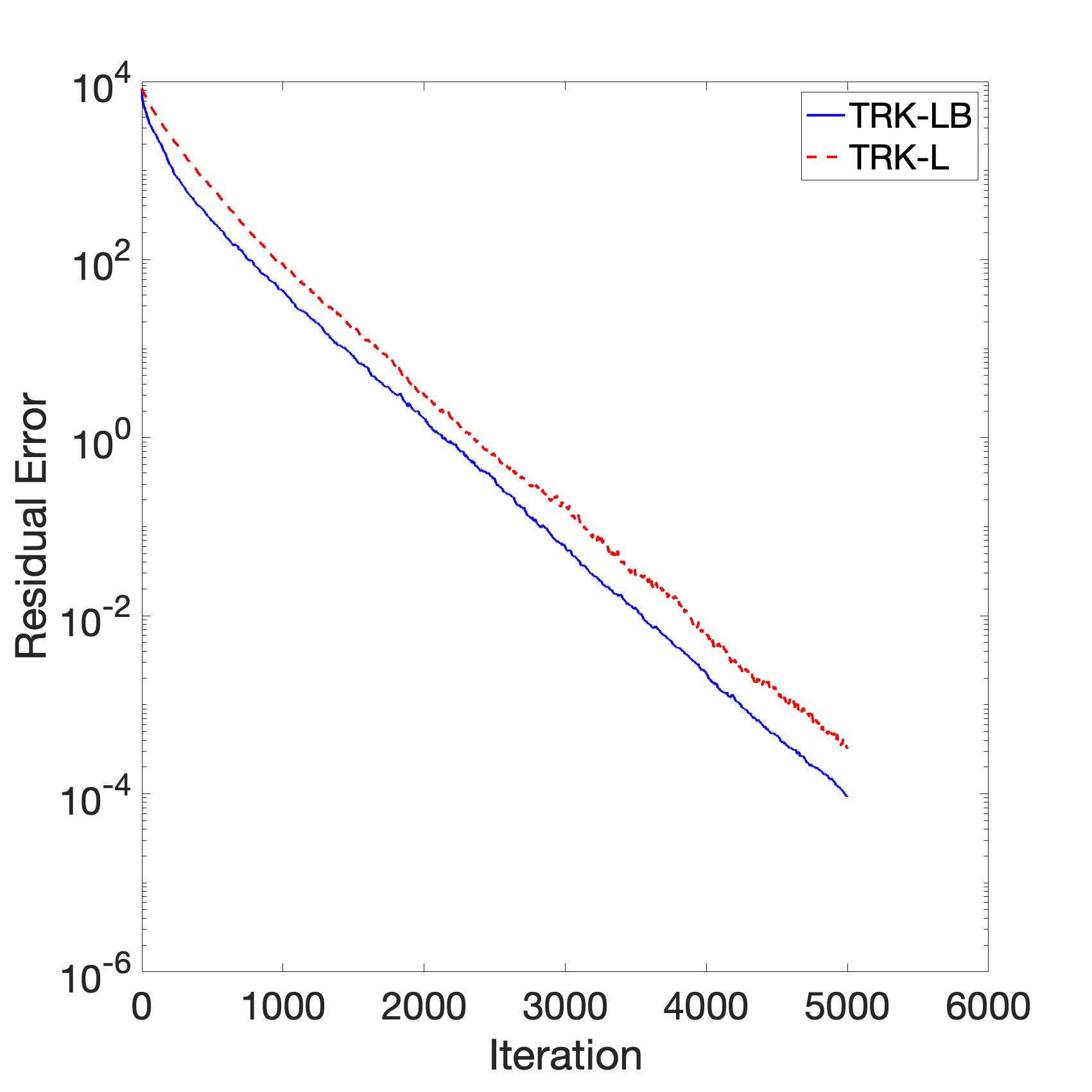}
  \hspace{2pt}
  \includegraphics[width=0.4\textwidth]{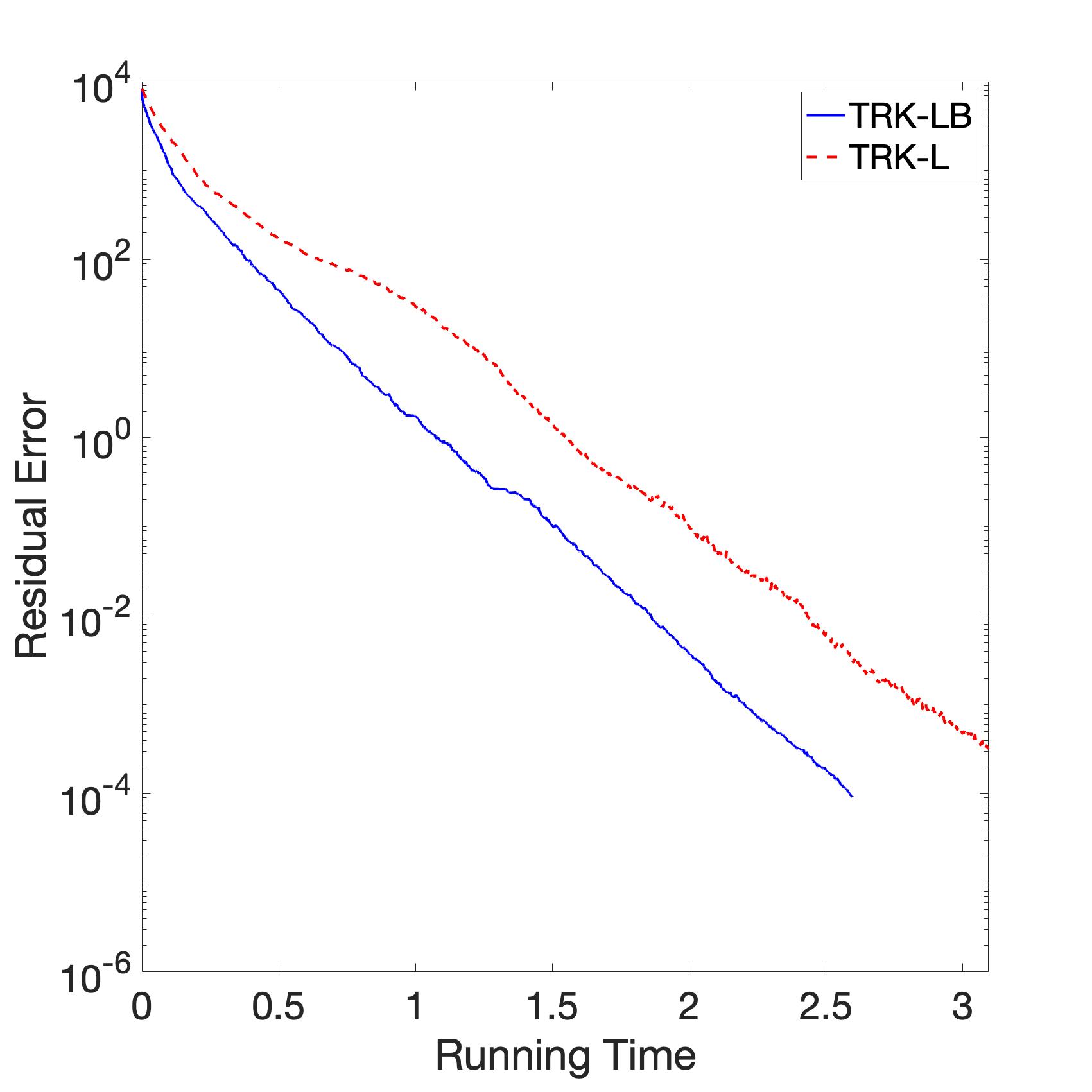}
  \caption{Comparison of TRK-LB and TRK-L for a Gaussian random system with $\A\in\mathbb R^{100\times 50\times 10}$,
  $\B\in\mathbb R^{100\times 7\times 10}$ and $\tilde \B\in\mathbb R^{50\times 7\times 10}$.
  Left:  Iterations versus residual errors of TRK-LB and TRK-L with step coefficient $\alpha=1.8$.
  Right:  Running time versus residual errors of TRK-LB and TRK-L with step coefficient $\alpha=1.8$.}
  \label{fig:tRKLB_vs_tRKL}
\end{figure}

\subsection{Applications in image deblurring}
In this section, we consider the applications of TRK-L and TRK-LB in image deblurring. 
A basic model of image blurring is given by 
$\A\ast\X=\B$,
where $\A$ denotes the blurring operator and $\B$ represents the blurry (and possibly noisy) image \cite{kilmer2013third}.
We consider two settings of image deblurring, one where the input image is the blurry image and we solve the linear system $\A\ast\X=\B$, and one where only a noisy observation $\tilde\B$ is available. 
In the first setting, the problem can be formulated as a problem of the form 
\[
\A\ast\X=\B, ~\X\ge 0,
\]
which can be solved using the TRK-LB method in Algorithm~\ref{alg:trknc}.
In the second setting, we consider a noisy observation  $\tilde\B$ with a bounded noise, and the deblurring problem is formulated as solving for $\X$ satisfying
\[
\B-\epsilon\le\A\ast\X \le \B+\epsilon, ~\X \ge 0,\\
\]
for a small $\epsilon>0$ depending on the noise level. The TRK-L method can be applied in this context. 

The experiments were performed on the 3D MRI image dataset \texttt{mri} in Matlab,
which consists of $12$ frames of size $128\times 128$ from an MRI data scan of a human cranium. The blurry images are generated by convolving the original images with a Gaussian convolution kernel of size $5\times 5$ with standard deviation 2. For the noisy setting, the noise tolerance $\epsilon$ was set as $0.2\approx 0.01\cdot\textrm{mean}(\B)$. We obtain a noisy observation $\tilde \B$ by perturbing the exact observation with a bounded noise:
\[
\tilde\B = \B+\N,
\]
where entries of the noise $\N$ are drawn independently from a uniform distribution on $[-0.2,0.2]$.
The step size is chosen to be \[
t_i= 2\normF{\A_{i::}}^2/(\max_{1\le j\le n}{\normF{\fft{\A_{i::}}_j}^2}),
\]
for each $i=1,\cdots,m.$

Figure \ref{fig:mri} displays the reconstructed images in both settings for various choices of initializations $\X^0$. Figure \ref{fig:reserror} 
showcases the residual error $\|\A\X^{(k)}-\B\|_F^2$ 
as a function of the iteration. 

\begin{figure}[htbp]
  \centering
  \includegraphics[width=0.4\textwidth]{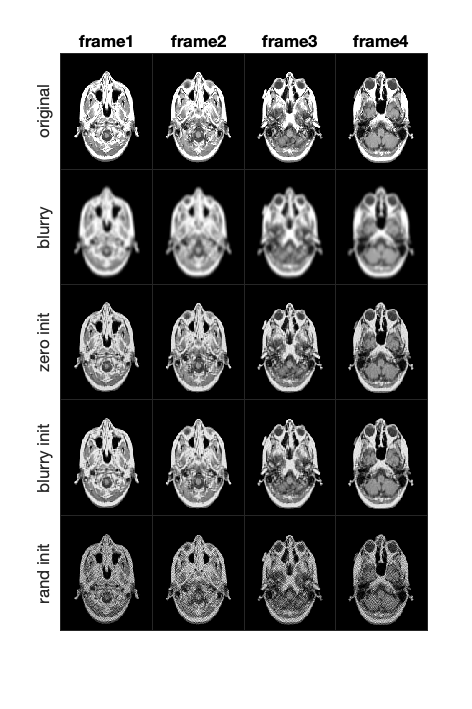}
  \includegraphics[width=0.4\textwidth]{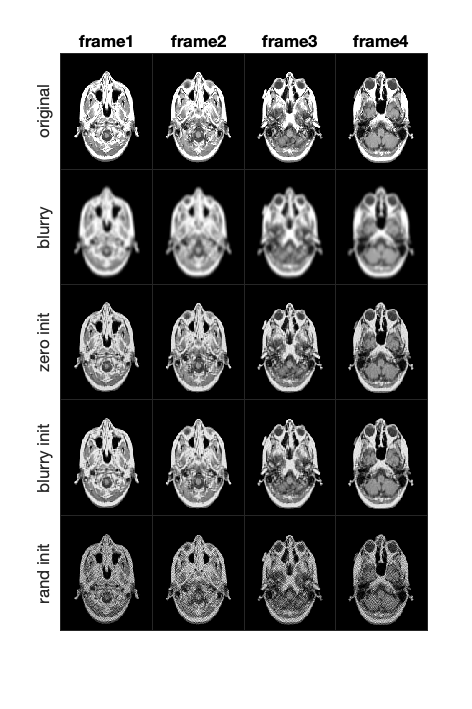}
  \caption{Results of image deblurring with exact observation (left) and noisy observation (right). Each column is a different frame of the image. The rows from top to bottom denote the original image, the blurry input image, and the reconstruction from a zero tensor initialization, the blurry image initialization, and a random  initialization generated from a Gaussian distribution with standard deviation 88. The noise tolerance $\epsilon$ was set as $0.2$.}
  \label{fig:mri}
\end{figure}

\begin{figure}[htbp]
  \centering
  \includegraphics[width=0.4\textwidth]{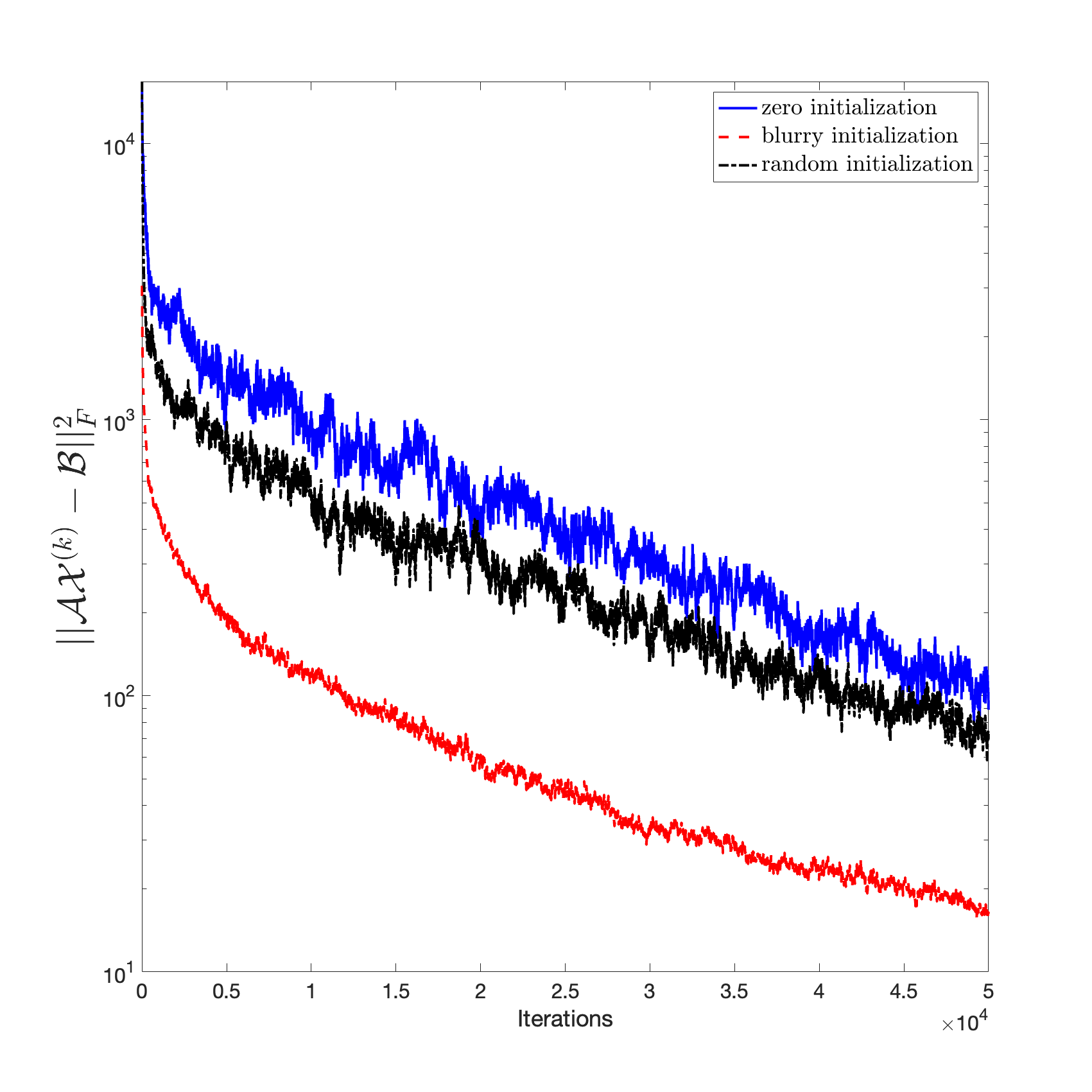}
  \hspace{2pt}
  \includegraphics[width=0.4\textwidth]{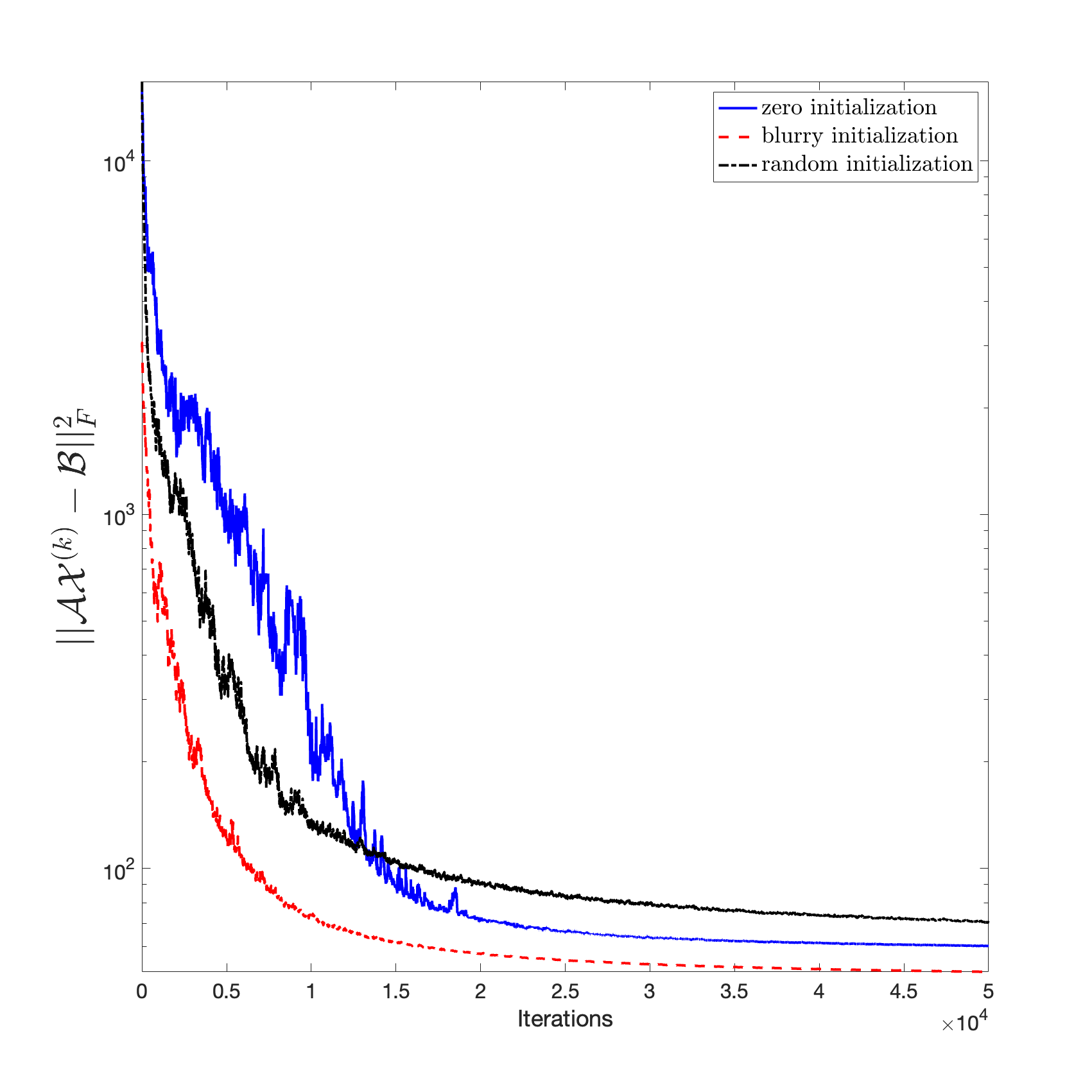}
  \caption{Residual errors of image deblurring with exact observation (left) and noisy observation (right). The iterates were derived from a zero
  tensor initialization (blue), the blurry image initialization (red), and a random initialization (black) generated from a Gaussian distribution with standard deviation 88. The noise tolerance $\epsilon$ was set as $0.2$.}
  \label{fig:reserror}
\end{figure}

 \section{Conclusions}\label{sec:conclude}
In this work, we proposed a block variant of RK method, B-MRK, for linear feasibility problems defined by matrices. We also proposed a TRK-L
method and a TRK-LB method for tensor linear feasibility problems defined under the t-produce. The B-MRK method and the TRK-L method solve general 
problems that involve mixed equality and inequality constraints, and the TRK-LB method is specifically tailored for 
problems that involve only equality constraints and bound constraints on variables. We show that all these three methods converge linearly in expection 
to the feasible region. Connections between TRK-L and B-MRK are made both by analysis and by numerical experiments.
The effectiveness of the three methods is demonstrated through numerical experiments
on a variety of Gaussian random systems and applications in image deblurring.

\section*{Declarations}

\noindent\textbf{Funding}~
JH was partially supported by NSF DMS \#2211318. DN was partially supported by NSF DMS \#2011140.

\noindent\textbf{Conflict of interest}~
On behalf of all authors, the corresponding author states that there is no conflict of interest. 

\bibliography{ref}
\end{document}